\documentclass{amsart}
\RequirePackage[utf8]{inputenc}

\usepackage{amsmath,amsthm,amsfonts,amssymb,amscd,amsbsy,dsfont,hyperref}
\usepackage[all]{xy}

\newcommand{\dd}{\mathrm{d}}
\newcommand{\id}{\operatorname{id}}
\newcommand{\vol}{\operatorname{vol}}

\newcommand{\Z}{\mathds Z}
\newcommand{\R}{\mathds R}
\newcommand{\C}{\mathds C}
\newcommand{\Hr}{\mathds H}
\newcommand{\Ca}{\mathds{C}\mathrm{a}}
\newcommand{\SO}{\mathsf{SO}}
\renewcommand{\O}{\mathsf O}
\newcommand{\SU}{\mathsf{SU}}
\newcommand{\U}{\mathsf{U}}
\newcommand{\Sp}{\mathsf{Sp}}
\newcommand{\Spin}{\mathsf{Spin}}
\newcommand{\G}{\mathsf{G}}
\newcommand{\K}{\mathsf{K}}
\renewcommand{\H}{\mathsf{H}}
\newcommand{\Sym}{\operatorname{S}(\Lambda^2 V)}
\renewcommand{\b}{\mathfrak b}
\newcommand{\g}{\mathrm g}
\newcommand{\tr}{\operatorname{tr}}
\newcommand{\Gr}{\operatorname{Gr}_2}
\newcommand{\m}{\mathfrak m}
\newcommand{\p}{\mathfrak p}
\newcommand{\h}{\mathfrak h}
\renewcommand{\k}{\mathfrak k}

\renewcommand{\Re}{\operatorname{Re}}
\newcommand{\diag}{\operatorname{diag}}

\allowdisplaybreaks

\hypersetup{
    pdftoolbar=true,
    pdfmenubar=true,
    pdffitwindow=false,
    pdfstartview={FitH},
    pdftitle={Strongly positive curvature},
    pdfauthor={Renato G. Bettiol and Ricardo A. E. Mendes},
    pdfsubject={AMS codes: 53B20, 53C20, 53C21, 53C30, 53C35},
    pdfkeywords={}
    pdfnewwindow=true,
    colorlinks=true, %set this false for printable version
    linkcolor=blue,
    citecolor=blue,
    urlcolor=black,
}

\newtheorem{theorem}{Theorem}[]
\newtheorem{lemma}[theorem]{Lemma}
\newtheorem{proposition}[theorem]{Proposition}

\newtheorem{mainthm}{\sc Theorem}

\newtheorem{problem}{\sc Problem}

\theoremstyle{definition}
\newtheorem{definition}[theorem]{Definition}

\theoremstyle{remark}
\newtheorem{remark}[theorem]{Remark}
\newtheorem{example}[theorem]{Example}

\title[Strongly positive curvature]{Strongly positive curvature}
\author[R. G. Bettiol]{Renato G. Bettiol}
\author[R. Mendes]{Ricardo A. E. Mendes}

\address{
\noindent University of Notre Dame \hfill\break\hfill\indent Department of Mathematics \hfill\break\hfill\indent 255 Hurley Building \hfill\break\hfill\indent Notre Dame, IN, 46556-4618, USA \hfill\break\hfill\indent \emph{E-mail address}: {\tt rbettiol@nd.edu} \hfill\break\hfill\indent \emph{E-mail address}: {\tt rmendes@nd.edu}
}

\allowdisplaybreaks
\numberwithin{equation}{section}
\numberwithin{theorem}{section}

\thanks{The first named author is partially supported by the NSF grant DMS-1209387, USA}

\date{\today}

\begin{document}
\begin{abstract}
We begin a systematic study of a curvature condition (strongly positive curvature) which lies strictly between positive curvature operator and positive sectional curvature, and stems from the work of Thorpe~\cite{thorpeJDG}. We prove that this condition is preserved under Riemannian submersions and Cheeger deformations, and that most compact homogeneous spaces with positive sectional curvature satisfy it.
\end{abstract}

\maketitle
%\tableofcontents

\section{Introduction}
Manifolds with positive curvature are a classic area of research in Riemannian Geometry. Nevertheless, very few topological obstructions to positive sectional curvature ($\sec>0$) are known, and many conjectures about this class remain elusive. In addition, the construction of new examples is a notoriously difficult problem: spheres and projective spaces remain the only known examples in dimensions $\geq25$. 
While there are many open questions about $\sec>0$, the subclass of manifolds with positive curvature operator ($R>0$) is now completely understood. In a ground-breaking paper, B\"ohm and Wilking~\cite{bw} used Ricci flow to prove that such manifolds also admit a metric with constant curvature, and are hence a finite quotient of a sphere. Given this current disparity, it is natural to investigate intermediate conditions between $R>0$ and $\sec>0$, with the hope of advancing the understanding of the latter. 

In this paper, we begin a systematic study of one such intermediate condition, which we call \emph{strongly positive curvature} following \cite{p2}. This condition stems from the work of Thorpe~\cite{singerthorpe,thorpeJDG,thorpe} and Zoltek~\cite{zoltek} in the 1970s, and has been implicitly studied by other authors, including P\"uttmann~\cite{puttmann}, who computed optimal pinching constants of homogeneous spaces, and  Grove, Verdiani and Ziller~\cite{p2} and Dearricott~\cite{de}, who recently verified the existence of a new closed $7$-manifold with $\sec>0$. Nevertheless, it is our feeling that manifolds with strongly positive curvature have not yet received the deserved attention \emph{by themselves}. The main goal of this paper is to rectify this by establishing the foundations for the study of this curvature condition and using them to analyze homogeneous examples, as well as indicating many problems to provide further directions that should be explored.

In order to define strongly positive curvature and state our results, recall that at each point of a Riemannian manifold $(M,\g)$, the \emph{curvature operator} is the self-adjoint linear operator $R\colon\Lambda^2 T_pM\to\Lambda^2 T_pM$ induced by the curvature tensor, see \eqref{eq:curvop}. Any $2$-plane $\sigma\subset T_pM$ can be viewed as a unit \emph{decomposable} element $X\wedge Y\in\Lambda^2 T_pM$, by choosing orthonormal vectors $X$ and $Y$ that span $\sigma$. As such, the \emph{sectional curvature} of planes tangent at $p\in M$ is computed by the quadratic form $\sec(\sigma)=\langle R(\sigma),\sigma\rangle$ on the subset of unit decomposable elements of $\Lambda^2 T_pM$. This subset is identified with the Grassmannian of (oriented) $2$-planes of $T_pM$.

Notice that any $4$-form $\omega$ naturally induces a self-adjoint operator $\omega\colon\Lambda^2 T_pM\to\Lambda^2T_pM$, given by $\langle\omega(\alpha),\beta\rangle=\langle\omega,\alpha\wedge\beta\rangle$, whose associated quadratic form clearly vanishes on decomposable elements $\sigma\in\Lambda^2T_pM$, since $\sigma\wedge\sigma=0$. The fundamental observation of Thorpe~\cite{thorpeJDG} is that sectional curvatures can be hence computed using the \emph{modified curvature operator} $R+\omega$; more precisely,
\begin{equation*}
\sec(\sigma)=\langle R(\sigma),\sigma\rangle=\langle (R+\omega)(\sigma),\sigma\rangle.
\end{equation*}
The crucial advantage is that the \emph{linear operator} $R+\omega$ is a much simpler object than the \emph{nonlinear function} $\sec$. In particular, if there exists $\omega$ such that $R+\omega$ is a positive-definite operator, then all planes automatically have positive sectional curvature. We say that the manifold $(M,\g)$ has \emph{strongly positive curvature} if this property is satisfied at every $p\in M$. Thus, it is clear from the definitions that
\begin{equation}\label{eq:implications}
R>0 \; \Longrightarrow \; \text{strongly positive curvature} \; \Longrightarrow \; \sec>0,
\end{equation}
and, as we shall see, none of the reverse implications hold in dimensions $\geq5$.

The first part of our paper regards basic properties of strongly positive curvature and operations that preserve it. Riemannian submersions $\pi\colon (\overline M,\overline \g)\to (M,\g)$ are a fundamental tool to construct manifolds with $\sec>0$, due to the fact that if $(\overline M,\overline \g)$ has $\sec>0$, then so does $(M,\g)$. A central result in our theory is that, although the analogous implication does not hold for $R>0$, it does hold for strongly positive curvature:

\begin{mainthm}\label{thm:A}
Let $\pi\colon (\overline M,\overline \g)\to (M,\g)$ be a Riemannian submersion. If $(\overline M,\overline \g)$ has strongly positive curvature, then $(M,\g)$ also has strongly positive curvature.
\end{mainthm}

It should come as no surprise that the proof of this result uses the Gray-O'Neill formula for curvature operators. The key observation is that a rearrangement of this formula in terms of the Bianchi identity yields a much more transparent presentation (see Subsection~\ref{subsec:submersions} for details):
\begin{align*}
\langle R(X\wedge Y),Z\wedge W\rangle &=\langle \overline R(\overline X\wedge\overline Y),\overline Z\wedge\overline W\rangle + 3\langle\alpha(\overline X\wedge\overline Y),\overline Z\wedge\overline W\rangle \nonumber\\
&\quad -3\b(\alpha)(\overline X,\overline Y,\overline Z,\overline W).
\end{align*}
Here, $\alpha=A^*A$ is a positive-semidefinite operator obtained from the tensor $A$ of the submersion, and $\b(\alpha)$ is the component of $\alpha$ orthogonal to tensors that satisfy the Bianchi identity. It turns out that such component is precisely the projection of $\alpha$ onto the subspace of operators induced by $4$-forms.

This technique is at the core of many of our results. For instance, it allows us to prove that strongly positive curvature is also preserved under Cheeger deformations:

\begin{mainthm}\label{mainthm:cheeger}
Suppose that $(M,\g)$ has strongly positive curvature and an isometric action of a compact Lie group $\G$. Then, the corresponding Cheeger deformation $(M,\g_t)$ also has strongly positive curvature for all $t>0$.
\end{mainthm}

The second part of our paper uses the above tools to study which manifolds with $\sec>0$ also admit strongly positive curvature. It follows from Theorem~\ref{thm:A} that all manifolds admitting a Riemannian submersion from a round sphere $S^n\to M$, such as projective spaces, have strongly positive curvature.\footnote{More precisely, a $4$-form proportional to the square of the K\"ahler form can be used to modify the curvature operator of $\C P^n$ to become positive-definite (analogously for $\Hr P^n$), see Remark~\ref{rem:kahlerform}.} The projective spaces $\C P^n$ and $\Hr P^n$, $n\geq2$, are hence examples of manifolds with strongly positive curvature that do not have positive curvature operator, showing that the converse to the first implication in \eqref{eq:implications} does not hold. A counter-example to the converse of the second implication in \eqref{eq:implications} is given by the Cayley plane $\Ca P^2$, which is the only compact rank one symmetric space (CROSS) that does not admit a submersion from a sphere. In fact, $\Ca P^2$ does not admit any homogeneous metrics with strongly positive curvature by a Hodge theory argument (Proposition~\ref{prop:cross2}). Although \emph{algebraic} counter-examples were previously known; to our knowledge, we provide the first examples on closed manifolds (see Remark~\ref{rem:zoltek}). Besides $\Ca P^2$, also the Berger space $B^{13}$ and certain Berger metrics on $S^{4n+3}$ have $\sec>0$ but do not have strongly positive curvature, see Remark~\ref{rem:b13} and Appendix~\ref{sec:berger}.

The first examples of closed manifolds with $\sec>0$ different from a CROSS were found among homogeneous spaces. The complete classification of closed simply-connected homogeneous spaces that admit an invariant metric with $\sec>0$ (see Theorem~\ref{thm:homsp}) was obtained in even dimensions by Wallach~\cite{wa} and in odd dimensions by B\'erard-Bergery~\cite{bb}. A remarkable feature of these examples is that, with only one exception,\footnote{The Berger space $B^7=\SO(5)/\SO(3)$, see Subsection~\ref{subsec:b7}.} they are the total space of a \emph{homogeneous fibration} $\K/\H\to\G/\H\to\G/\K$, where $\H\subset\K\subset\G$ are compact Lie groups. Sufficient conditions to produce an invariant metric with $\sec>0$ on $\G/\H$ were given by Wallach~\cite{wa} (see also \cite{es} and \cite[Prop. 4.3]{zillersurvey}), among which is \emph{fatness} of the bundle $\K/\H\to\G/\H\to\G/\K$.
We find analogous conditions (including a notion of \emph{strong fatness}, see Subsection~\ref{subsec:strongfat}), which are sufficient to produce invariant metrics with \emph{strongly} positive curvature on $\G/\H$ (see Theorem~\ref{thm:wallach}). Combining this with the above mentioned classification of homogeneous spaces with $\sec>0$ (Theorem~\ref{thm:homsp}), we obtain an essentially complete classification of closed simply-connected homogeneous spaces that admit an invariant metric with strongly positive curvature:

\begin{mainthm}\label{mainthm:homspaces}
The following homogeneous spaces admit an invariant metric with strongly positive curvature:
\begin{itemize}
\item Compact rank one symmetric spaces $S^n$, $\C P^n$ and $\Hr P^n$;
\item Wallach flag manifolds $W^6=\SU(3)/\mathsf T^2$ and $W^{12}=\Sp(3)/\Sp(1)\Sp(1)\Sp(1)$;
\item Aloff-Wallach spaces $W^7_{k,\ell}=\SU(3)/\mathsf S^1_{k,\ell}$;
\item Berger spaces $B^7=\SO(5)/\SO(3)$ and $B^{13}=\SU(5)/\Sp(2)\cdot\mathsf S^1$.
\end{itemize}
Moreover, the Cayley plane $\Ca P^2=\mathsf F_4/\Spin(9)$ does not admit any homogeneous metric with strongly positive curvature.
\end{mainthm}

Some of the above cases, namely $W^6$, $W^7_{1,1}$ and $B^7$, were already known to P\"uttmann~\cite{puttmann,puttmann2}.
The \emph{only} remaining homogeneous space that could support strongly positive curvature is the Wallach flag manifold $W^{24}=\mathsf F_4/\Spin(8)$, which is the total space of a homogeneous fibration $S^8\to W^{24}\to\Ca P^2$. From Theorem~\ref{thm:A}, any metric on $W^{24}$ for which the above is a Riemannian submersion does not have strongly positive curvature. Nevertheless, there are $\mathsf F_4$-invariant metrics on $W^{24}$ with $\sec>0$ for which this does not hold. One such metric could, in principle, have strongly positive curvature.

Besides the above homogeneous manifolds, the only other known examples\footnote{Apart from the proposed construction of a metric with $\sec>0$ on the Gromoll-Meyer exotic sphere, by Petersen and Wilhelm \cite{pw}.} of closed manifolds with $\sec>0$ are biquotients (\emph{Eschenburg} and \emph{Bazaikin spaces}) and the exotic $T_1 S^4$ studied in \cite{de,p2}. The latter was shown to have strongly positive curvature in \cite{p2}, while the former are not yet known to have this property.

Several natural questions as the above arise from our study of manifolds with strongly positive curvature. In Section~\ref{sec:openq}, we compile a list of such problems, including a few remarks about each of them. For instance, regarding attempts to find topological obstructions to strongly positive curvature, we observe that strongly positive curvature is \emph{not preserved under Ricci flow} and \emph{does not imply positivity of the Gauss-Bonnet integrand}.

This paper is organized as follows. In Section~\ref{sec:basics}, we study the basic properties of strongly positive curvature, including the proofs of Theorems~\ref{thm:A} and \ref{mainthm:cheeger}. In Section~\ref{sec:cross}, we analyze which of the CROSS have strongly positive curvature. Section~\ref{sec:strongfatwallach} introduces the notion of strong fatness for homogeneous fibrations and gives sufficient conditions for the existence of an invariant metric with strongly positive curvature on its total space. In Section~\ref{sec:homsp}, we complete the proof of Theorem~\ref{mainthm:homspaces} by verifying such conditions on the examples of homogeneous fibrations and directly computing the modified curvature operator of the exceptional case $B^7$. Section~\ref{sec:openq} contains a list of open problems related to strongly positive curvature, accompanied by a few remarks. Finally, Appendix~\ref{sec:berger} regards Berger metrics on spheres.

\bigskip
\noindent\textbf{Acknowledgements.} It is a pleasure to thank Karsten Grove, Thomas P\"uttmann, Luigi Verdiani and Wolfgang Ziller for their constant interest in this project and many valuable suggestions. We also thank Amy Buchmann and David Johnson for helpful conversations on related subjects.

\section{Definitions and Basic Properties}
\label{sec:basics}

In this section, we introduce the definition of strongly positive curvature, and study its basic properties. In order to do so, we recall some facts about curvature operators from the point of view of Thorpe~\cite{singerthorpe,thorpeJDG,thorpe}. Throughout the text, we follow the sign conventions in Besse~\cite{besse}; in particular, the curvature operator of a Riemannian manifold is given by:
\begin{equation}\label{eq:curvop}
\langle R(X\wedge Y),Z\wedge W\rangle =\langle R(X,Y)Z,W\rangle = \langle \nabla_{[X,Y]}Z-\nabla_X\nabla_Y Z+\nabla_Y\nabla_X Z,W\rangle.
\end{equation}

Since almost all known constructions of manifolds with $\sec>0$ use a combination of submersions and Cheeger deformations \cite{zillersurvey}, it is natural to expect that these also play a central role in constructing manifolds with strongly positive curvature. This turns out to be exactly the case, as explained in Theorems~\ref{thm:submersions} and \ref{thm:cheegerdef} (Theorems~\ref{thm:A} and \ref{mainthm:cheeger}), which play a central role in the remainder of this paper. We remark that Propositions~\ref{prop:immersions} and \ref{prop:ginvariance} were also observed by P\"uttmann~\cite{puttmann}.

\subsection{Preliminaries}
Let $V$ be a finite-dimensional vector space with an inner product $\langle\cdot,\cdot\rangle$, and $\Lambda^k V$ its $k^{th}$ exterior power. By means of the induced inner product, we identify $\Lambda^k V$ with the space $\Lambda^k V^*$ of $k$-forms on $V$. Denote by $\Sym$ the space of linear self-adjoint operators $R\colon\Lambda^2 V\to\Lambda^2 V$, with the inner product $\langle R,S\rangle=\tr RS$. Any $\omega\in\Lambda^4 V$ can be viewed as an operator $\omega\in\Sym$ via
\begin{equation}\label{eq:wrw}
\langle \omega(X\wedge Y),Z\wedge W\rangle:=\omega(X,Y,Z,W).
\end{equation}
This gives an isometric immersion $\Lambda^4 V \subset\Sym$. The orthogonal projection $\b\colon\Sym\to\Lambda^4V$ onto this subspace is given by the \emph{Bianchi map}:
\begin{equation*}
\b(R)(X,Y,Z,W)=\tfrac13\Big(\langle R(X\wedge Y),Z\wedge W\rangle + \langle R(Y\wedge Z),X\wedge W\rangle + \langle R(Z\wedge X),Y\wedge W\rangle\Big).
\end{equation*}
In particular, $\Sym=\ker\b\oplus\Lambda^4 V$ is an orthogonal direct sum decomposition.\footnote{Moreover, this is a decomposition of $\O(n)$-representations, since the Bianchi map is equivariant with respect to the natural $\O(n)$-actions. The $\O(n)$-representation on $\Lambda^4 V$ is irreducible, while the $\O(n)$-representation on $\ker\b$ splits as sum of three irreducible subrepresentations, that are related to the Ricci and Weyl tensors, see \cite[p. 357]{singerthorpe}.} The elements $R\in\ker\b$ are called \emph{algebraic curvature operators}, since they are self-adjoint operators $R\colon\Lambda^2 V\to\Lambda^2 V$ that satisfy the first Bianchi identity, as the curvature operator of a Riemannian manifold.

The Grassmannian of (oriented) $2$-planes in $V$ can be seen as the subset $\Gr(V)\subset\Lambda^2V$ of decomposable $2$-vectors with unit norm, by identifying $\sigma=X\wedge Y\in \Lambda^2V$ with the $2$-plane spanned by $X$ and $Y$. Given $R\in\Sym$, the restriction of the associated quadratic form to $\Gr(V)$ is called its \emph{sectional curvature function}:
\begin{equation}\label{eq:secr}
\sec_R\colon\Gr(V)\longrightarrow\R, \quad \sec_R(\sigma):=\langle R(\sigma),\sigma\rangle.
\end{equation}
It is easy to see that $\sigma\in\Lambda^2 V$ is decomposable if and only if $\langle \omega(\sigma),\sigma\rangle=0$ for all $\omega\in\Lambda^4 V$, see \cite[Prop 1.3]{thorpeJDG}. As a result, \eqref{eq:secr} only depends on the component of $R$ in $\ker\b$. More precisely, given any $\omega\in\Lambda^4V$ and $\sigma\in\Gr(V)$, we have
\begin{equation}\label{eq:secrw}
\sec_{R+\omega}(\sigma)=\langle (R+\omega)(\sigma),\sigma\rangle=\langle R(\sigma),\sigma\rangle+\langle \omega(\sigma),\sigma\rangle=\sec_R(\sigma).
\end{equation}
Given an algebraic curvature operator $R$, although the \emph{modified curvature operator} $R+\omega$ no longer satisfies the first Bianchi identity, both have the same sectional curvature function.

\begin{definition}[Strongly positive curvature]\label{def:strongpos}
An algebraic curvature operator $R\colon\Lambda^2 V\to\Lambda^2 V$ has \emph{strongly positive curvature} if there exists $\omega\in\Lambda^4 V$ such that $(R+\omega)\colon\Lambda^2 V\to\Lambda^2 V$ is a positive-definite operator. A Riemannian manifold $(M,\g)$ has \emph{strongly positive curvature} if the curvature operator $R_p\colon\Lambda^2 T_pM\to\Lambda^2 T_pM$ has strongly positive curvature for all $p\in M$.
\end{definition}

In particular, it follows from \eqref{eq:secrw} that if $(M,\g)$ has strongly positive curvature, then it also has $\sec>0$, i.e., the sectional curvature of any $2$-plane tangent to $(M,\g)$ is positive. It is also clear that if $R>0$, i.e., the curvature operator of $(M,\g)$ is positive-definite, then $(M,\g)$ has strongly positive curvature. Moreover, it was known since Zoltek~\cite{zoltek}, see also \cite{thorpeJDG}, that there exist algebraic curvature operators in dimensions $\geq5$ that have $\sec_R>0$ but cannot be modified with any $4$-form to become positive-definite, cf. Remark~\ref{rem:zoltek}. In the sequel, we provide many examples in dimension $\geq4$ of closed manifolds with strongly positive curvature and not diffeomorphic to spheres, which hence cannot have $R>0$ by \cite{bw}. Thus, strongly positive curvature is truly an \emph{intermediate} curvature condition between $\sec>0$ and $R>0$, in dimensions $\geq5$. However, in dimensions $\leq4$, Thorpe~\cite{thorpe} proved the following converse result, see also \cite[Prop 3.4]{puttmann}.

\begin{proposition}\label{prop:dim4}
If $\dim V\leq4$, an algebraic curvature operator $R$ has strongly positive curvature if and only if $\sec_R>0$.
\end{proposition}

\begin{remark}\label{rem:smoothness}
From Definition~\ref{def:strongpos}, strongly positive curvature implies\footnote{Note that, for each strongly positive algebraic curvature operator $R$, the set of $\omega$'s such that $R+\omega$ is positive-definite is bounded and convex. Its center of mass provides a continuous map $R\mapsto \omega_R$ such that $R+\omega_R$ is positive-definite, by the Dominated Convergence Theorem.}
the existence of a \emph{continuous} $4$-form $\omega\in\Omega^4(M)$ such that $R_p+\omega_p$ is positive-definite for all $p\in M$. A priori, such a $4$-form does not have to be \emph{smooth}, but it can be easily seen (from the openness of the positivity condition) that it admits a small perturbation $\omega'$ which is smooth and such that $R_p+\omega'_p$ is positive-definite for all $p\in M$.
\end{remark}

\subsection{Submersions}\label{subsec:submersions}
A fundamental result is that strongly positive curvature, similarly to $\sec>0$, is preserved under Riemannian submersions (Theorem~\ref{thm:A}):

\begin{theorem}\label{thm:submersions}
Let $\pi\colon (\overline M,\overline \g)\to (M,\g)$ be a Riemannian submersion. If $(\overline M,\overline \g)$ has strongly positive curvature, then $(M,\g)$ also has strongly positive curvature.
\end{theorem}

\begin{proof}
Given that strongly positive curvature is a pointwise condition, choose $p\in M$ and $\overline p\in\overline M$ such that $\pi(\overline p)=p$, and set $V=T_pM$ and $\overline V=T_{\overline p}\overline M$. For any $X\in V$, we denote by $\overline X\in\overline V$ its horizontal lift and consider the inclusion map $i\colon V\hookrightarrow\overline V$, $i(X)=\overline X$, through which we identify $V$ with a subspace of $\overline V$. We denote by $V^\perp$ the orthogonal complement of this subspace. The $A$-tensor of the submersion, given by $A_X Y=\tfrac12[\overline X,\overline Y]^{\mathcal V}$, induces a skew-symmetric map $A\colon V\times V\to V^\perp$, which we can interpret as $A\colon\Lambda^2 V\to V^\perp$. Set $\alpha=A^*A\in\Sym$, i.e.,
\begin{equation}\label{eq:alpha}
\langle\alpha(X\wedge Y),Z\wedge W\rangle=\langle A_X Y,A_Z W\rangle.
\end{equation}
Clearly, $\alpha\colon\Lambda^2 V\to\Lambda^2 V$ is a positive-semidefinite operator. From the Gray-O'Neill formulas, see \cite[Thm 9.28f]{besse}, we have
\begin{align*}
\langle R(X\wedge Y),Z\wedge W\rangle &=\langle \overline R(\overline X\wedge\overline Y),\overline Z\wedge\overline W\rangle +2\langle A_{\overline X} \overline Y,A_{\overline Z} \overline W\rangle \\ 
&\quad -\langle A_{\overline Y} \overline Z, A_{\overline X}\overline W\rangle +\langle A_{\overline X}\overline Z,A_{\overline Y}\overline W\rangle \\
&=\langle \overline R(\overline X\wedge\overline Y),\overline Z\wedge\overline W\rangle + 3\langle\alpha(\overline X\wedge\overline Y),\overline Z\wedge\overline W\rangle\\
&\quad -\langle\alpha(\overline Y\wedge\overline Z),\overline X\wedge\overline W\rangle-\langle\alpha(\overline Z\wedge\overline X),\overline Y\wedge\overline W\rangle\\
&\quad -\langle\alpha(\overline X\wedge\overline Y),\overline Z\wedge\overline W\rangle\\
&=\langle \overline R(\overline X\wedge\overline Y),\overline Z\wedge\overline W\rangle + 3\langle\alpha(\overline X\wedge\overline Y),\overline Z\wedge\overline W\rangle\\
&\quad -3\b(\alpha)(\overline X,\overline Y,\overline Z,\overline W).
\end{align*}
Thus, if there exists $\overline\omega\in\Lambda^4\overline V$ such that $\overline R+\overline\omega$ is positive-definite, it follows that $R+\omega$ is positive-definite for $\omega=i^*\overline\omega+3\b(\alpha)\in\Lambda^4 V$.
\end{proof}

The above way of rewriting the Gray-O'Neill formula for curvature operators as
\begin{align}\label{eq:oneill}
\langle R(X\wedge Y),Z\wedge W\rangle &=\langle \overline R(\overline X\wedge\overline Y),\overline Z\wedge\overline W\rangle + 3\langle\alpha(\overline X\wedge\overline Y),\overline Z\wedge\overline W\rangle \nonumber\\
&\quad -3\b(\alpha)(\overline X,\overline Y,\overline Z,\overline W),
\end{align}
where $\alpha$ is given by \eqref{eq:alpha}, seems more natural than its conventional presentation. In particular, from $3\b(\alpha)\in\Lambda^4 V$ and \eqref{eq:secrw}, we immediately recover the well-known
\begin{equation}\label{eq:oneillsec}
\sec(X,Y)=\overline{\sec}(\overline X,\overline Y)+3\|A_X Y\|^2,
\end{equation}
which shows that $\sec>0$ is preserved under Riemannian submersions. To illustrate how this approach makes curvature operator computations easier, we apply it to normal homogeneous spaces to obtain formula \eqref{eq:curvhomsp}, which will be useful to study homogeneous fibrations in Section~\ref{sec:strongfatwallach}.

\begin{example}\label{ex:homsp}
Let $\G/\H$ be a homogeneous space with a \emph{normal} homogeneous metric. Denote by $Q$ the bi-invariant metric on $\G$ and by $\mathfrak g$ and $\h$ the Lie algebras of $\G$ and $\H$ respectively. Let $\h^\perp$ be the $Q$-orthogonal complement of $\h\subset\mathfrak g$, so that we may identify $\h^\perp\cong T_{(e\H)}\G/\H$ and metric on $\G/\H$ with $Q|_{\h^\perp}$. In this situation, the horizontal lift of $X\in\h^\perp$ is simply its inclusion $X\in\mathfrak g$ and the vertical projection $X^{\mathcal V}$ is given by its component $X_\h$ in $\h$. Thus, \eqref{eq:alpha} gives
\begin{equation}\label{eq:alphahomsp}
\big\langle\alpha_{\G/\H}(X\wedge Y),Z\wedge W\big\rangle=\tfrac14Q\big([X,Y]_\h,[Z,W]_\h\big),
\end{equation}
and hence the curvature operator of $(\G/\H,Q|_{\h^\perp})$ is given by
\begin{equation}\label{eq:curvhomsp}
\begin{aligned}
\big\langle R_{\G/\H}(X\wedge Y),Z\wedge W\big\rangle &=\tfrac14Q([X,Y],[Z,W])+\tfrac34Q\big([X,Y]_\h,[Z,W]_\h\big) \\
&\quad -3\b(\alpha_{\G/\H})(X,Y,Z,W).
\end{aligned}
\end{equation}
It is hence clear that the operator $R_{\G/\H}\colon\Lambda^2\h^\perp\to\Lambda^2\h^\perp$ can be modified with the $4$-form $3\b(\alpha_{\G/\H})\in\Lambda^4\h^\perp$ to become positive-semidefinite.
\end{example}

\subsection{Totally geodesic immersions}
Analogously to $\sec>0$, by the Gauss formula, totally geodesic submanifolds also inherit strongly positive curvature:

\begin{proposition}\label{prop:immersions}
Let $i\colon (M,\g)\to (\overline M,\overline \g)$ be a totally geodesic immersion. If $(\overline M,\overline \g)$ has strongly positive curvature, then $(M,\g)$ also has strongly positive curvature.
\end{proposition}
%\begin{proof}
%Given $p\in M$, set $V=T_pM$ and $\overline V=T_{i(p)} \overline M$. For any $X\in V$, we write $\overline X=\dd i(p)X\in\overline V$. Since $i\colon M\to\overline M$ is totally geodesic, from the Gauss formula \cite[Thm 1.72c]{besse}, we have
%\begin{equation}
%\langle R(X\wedge Y),Z\wedge W\rangle =\langle \overline R(\overline X\wedge\overline Y),\overline Z\wedge\overline W\rangle.
%\end{equation}
%Thus, if there exists $\overline\omega\in\Lambda^4\overline V$ such that $\overline R+\overline\omega$ is positive-definite, then its restriction $\omega=(\dd i(p))^*\overline\omega\in\Lambda^4 V$ is such that $R+\omega$ is positive-definite.
%\end{proof}

\subsection{\texorpdfstring{$\G$-invariance}{G-invariance}}
The following result is fundamental to study manifolds with strongly positive curvature and many symmetries. It follows from convexity and a standard averaging argument, see also P\"uttmann~\cite[Lemma 3.5]{puttmann}.

\begin{proposition}\label{prop:ginvariance}
Suppose that $(M,\g)$ has strongly positive curvature and an isometric action of a compact Lie group $\G$. Then, there exists a $\G$-invariant $\overline\omega\in\Omega^4(M)$ such that $R+\overline\omega$ is positive-definite (at all points in $M$).
\end{proposition}
%\begin{proof}
%Let $\omega\in\Omega^4(M)$ be such that $R+\omega$ is positive-definite and let $\overline\omega:=\int_\G g^*\omega\,\dd g$ be the result of averaging it with the $\G$-action. Since the $\G$-action is isometric, $R+\overline\omega=R+\int_\G g^*\omega\,\dd g=\int_\G g^*(R+\omega)\,\dd g$, which is positive-definite by convexity of the set of positive-definite operators.
%\end{proof}

\subsection{Cheeger deformations}
Let $(M,\g)$ be a Riemannian manifold with an isometric action by a compact Lie group $\G$. Its \emph{Cheeger deformation}, introduced in~\cite{cheeger}, is a path $\g_t$ of $\G$-invariant metrics on $M$ starting at $\g$, obtained by shrinking the orbits of the $\G$-action (at possibly different rates), see also \cite{mueter,zillermueter}. More precisely, let $Q$ be a bi-invariant metric on $\G$ and consider the product manifold $(M\times \G,\g+\tfrac1t Q)$. Then, $\g_t$ is defined as the metric on $M$ that makes
\begin{equation}\label{eq:submcheeger}
\pi\colon \left(M\times \G,\g+\tfrac1t Q\right)\longrightarrow (M,\g_t), \quad \pi(p,g)=g^{-1}p,
\end{equation}
a Riemannian submersion. Note that it automatically follows from \eqref{eq:oneillsec} that if $(M,\g)$ has $\sec\geq0$, so does $(M,\g_t)$, for any $t>0$.

In order to explicitly compute the curvature operator of $(M,\g_t)$, we need to establish some notation. Let $\mathfrak g$ be the Lie algebra of $\G$ and $\mathfrak g_p$ the Lie algebra of the isotropy $\G_p$, for each $p\in M$. Consider the $Q$-orthogonal splitting $\mathfrak g=\mathfrak g_p\oplus\m_p$, and identify $\m_p$ with the tangent space to the $\G$-orbit through $p$ via action fields, i.e., we identify each $X\in\m_p$ with $X^*_p=\tfrac{\dd}{\dd s}\exp(sX)p|_{s=0}$. This determines a $\g_t$-orthogonal splitting\footnote{We note that this $\g_t$-orthogonal splitting is actually independent of $t$, since vertical (respectively horizontal) directions of $\g$ remain vertical (respectively horizontal) for $\g_t$, $t>0$.} $T_pM=\mathcal{V}_p\oplus\mathcal{H}_p$ in \emph{vertical} and \emph{horizontal spaces}, respectively
\begin{equation*}
\mathcal{V}_p:=T_p \G(p)=\{X^*_p:X\in\mathfrak m_p\} \,\,\mbox{ and } \,\, \mathcal{H}_p:=\{v\in T_pM : \g_t(v,\mathcal{V}_p)=0\}.
\end{equation*}
For each $t\geq0$, let $P_t\colon\m_p\to\m_p$ and $C_t\colon T_pM\to T_pM$ be the symmetric automorphisms such that
\begin{align*}
Q(P_t(X),Y) &=\g_t(X^*_p,Y^*_p), \quad X,Y\in\mathfrak m_p \\
\g(C_t(X),Y) &=\g_t(X,Y), \quad X,Y\in T_pM.
\end{align*}
By computing the horizontal lift of $X$ under \eqref{eq:submcheeger}, it is not hard to prove that
\begin{equation}\label{eq:ptct}
\begin{aligned}
P_t(X) &=(P_0^{-1}+t\id)^{-1}(X)=P_0\,(\id+tP_0)^{-1}(X), \quad X\in \m_p, \\
C_t(X) &=P_0^{-1}P_t(X^\mathcal{V})+X^\mathcal{H}, \quad X\in T_pM,
\end{aligned}
\end{equation}
see \cite[Prop 1.1]{zillermueter}. The inverse operator $C_t^{-1}\colon T_pM\to T_pM$ is called \emph{Cheeger reparametrization}, and its use greatly simplifies computations. This is due to the fact that the horizontal lift $\overline X\in T_pM\times\mathfrak g$ of $X\in T_pM$ with respect to \eqref{eq:submcheeger} is quite complicated, while the horizontal lift of $C_t^{-1}X=(\id+tP_0)X^*_\mathfrak m+X^\mathcal H$ is simply
\begin{equation}
\overline{C_t^{-1}X}=\big(X,-tP_0X_\mathfrak m\big)\in T_pM\times\mathfrak g.
\end{equation}
After these observations, we are now ready to prove Theorem~\ref{mainthm:cheeger} on Cheeger deformations of manifolds with strongly positive curvature:

\begin{theorem}\label{thm:cheegerdef}
Suppose that $(M,\g)$ has strongly positive curvature and an isometric action of a compact Lie group $\G$. Then, the corresponding Cheeger deformation $(M,\g_t)$ also has strongly positive curvature for all $t>0$.
\end{theorem}

\begin{proof}
Denote by $C_t^{-1}(X\wedge Y):=C_t^{-1}X\wedge C_t^{-1}Y$ the map induced by the Cheeger reparametrization on $\Lambda^2 TM$. Applying \eqref{eq:oneill} in the above situation, we get that the curvature operator $R_t\colon\Lambda^2 T_pM\to\Lambda^2 T_pM$ of $(M,\g_t)$ is given by the formula
\begin{align*}
\left\langle R_t\big(C_t^{-1}(X\wedge Y)\big),C_t^{-1}(Z\wedge W)\right\rangle_t &=\langle R(X\wedge Y),Z\wedge W\rangle \\
&\quad +t^3 Q\big([P_0X_\m,P_0Y_\m],[P_0Z_\m,P_0W_\m]\big) \\
&\quad+3\left\langle\alpha\left(\overline{C_t^{-1}X}\wedge \overline{C_t^{-1}Y}\right),\overline{C_t^{-1}Z}\wedge \overline{C_t^{-1}W}\right\rangle \\
&\quad-3\b(\alpha)\left(\overline{C_t^{-1}X},\overline{C_t^{-1}Y},\overline{C_t^{-1}Z},\overline{C_t^{-1}W}\right),
\end{align*}
where $\alpha=A^*A$ is the positive-semidefinite operator given by \eqref{eq:alpha}. Similarly, the second term is a multiple of the quadratic form associated to the positive-semidefinite operator $L^*L$, where $L(X\wedge Y)=[P_0X_\m,P_0Y_\m]$. Thus, if there exists $\omega\in\Omega^4(M)$ such that $R+\omega$ is positive-definite, then setting $\omega_t\in\Omega^4(M)$ so that
\begin{align*}
\omega_t\left(\overline{C_t^{-1}X},\overline{C_t^{-1}Y},\overline{C_t^{-1}Z},\overline{C_t^{-1}W}\right)&=
\omega(X,Y,Z,W)\\
&\quad+3\b(\alpha)\left(\overline{C_t^{-1}X},\overline{C_t^{-1}Y},\overline{C_t^{-1}Z},\overline{C_t^{-1}W}\right),
\end{align*}
it follows that $R_t+\omega_t$ is positive-definite for all $t>0$.
\end{proof}

\section{Compact Rank One Symmetric Spaces}
\label{sec:cross}

The first and most well-known examples of closed manifolds with $\sec>0$ are the Compact Rank One Symmetric Spaces (CROSS). These are the manifolds $\G/\K$, where $(\G,\K)$ is a symmetric pair of rank one, and consist of:
\begin{itemize}
\item[(i)] Spheres $S^n=\SO(n+1)/\SO(n)$;
\item[(ii)] Complex projective spaces $\C P^n=\SU(n+1)/\mathsf{S}(\U(n)\U(1))$;
\item[(iii)] Quaternionic projective spaces $\Hr P^n=\Sp(n+1)/\Sp(n)$;
\item[(iv)] Cayley plane $\Ca P^2=\mathsf{F}_4/\Spin(9)$.
\end{itemize}
Denote respectively by $\k$ and $\mathfrak g$ the Lie algebras of $\K$ and $\G$, let $Q$ be a bi-invariant metric on $\G$, and $\mathfrak g=\k\oplus\m$ be a $Q$-orthogonal splitting. The standard Riemannian metric (round or Fubini-Study) is the homogeneous metric induced by $Q|_\m$, which is called \emph{normal homogeneous metric}. In the above cases, this metric is the \emph{unique} $\G$-invariant metric, since the corresponding isotropy representation is irreducible. %\footnote{Note that although the Lie algebra $\mathfrak{so}(4)\cong\mathfrak{su(2)}\oplus\mathfrak{su}(2)$ is not simple, and has a $2$-parameter family of bi-invariant metrics $Q_{a,b}$, the subalgebra $\k$ that corresponds to the symmetric space $\SO(4)/\SO(3)=S^3$ must be one of the irreducible summands. Thus, the normal homogeneous metric $Q_{a,b}|_\m$ only depends on one parameter, and is hence also a multiple of the round metric.}

\begin{remark}
Some of the above manifolds also admit homogeneous metrics with smaller isometry groups, including homogeneous metrics with $\sec>0$ \cite{vz,ziller}. However, unless otherwise mentioned, we henceforth assume that any sphere or projective space is endowed with the above standard metric.
\end{remark}

From Example~\ref{ex:homsp}, it follows that the curvature operator \eqref{eq:curvhomsp} of any compact normal homogeneous space can be modified with (an invariant) $4$-form to become positive-semidefinite. We now analyze in which cases among the CROSS it can be further modified to become positive-definite.

\begin{proposition}\label{prop:cross1}
The symmetric spaces $S^n$, $\C P^n$ and $\Hr P^n$ have strongly positive curvature.
\end{proposition}

\begin{proof}
The curvature operator of the unit round sphere $S^n$ is the identity map $\id\colon\Lambda^2 TS^n\to\Lambda^2 TS^n$, which is clearly positive-definite, hence $S^n$ has strongly positive curvature. The curvature operators of $\C P^n$ and $\Hr P^n$ are positive-semidefinite, but have nontrivial kernel. However, since the Hopf bundles
\begin{equation*}
S^1\longrightarrow S^{2n+1}\longrightarrow \C P^n \quad \mbox{ and }\quad S^3\longrightarrow S^{4n+3}\longrightarrow \Hr P^n
\end{equation*}
are Riemannian submersions, Theorem~\ref{thm:submersions} implies that $\C P^n$ and $\Hr P^n$ also have strongly positive curvature.
\end{proof}

\begin{remark}\label{rem:kahlerform}
In the case of $\C P^n$, the operator $\alpha$ (see \eqref{eq:alpha}) is given by $\omega_{\text{FS}}\otimes\omega_{\text{FS}}$, where $\omega_{\text{FS}}$ is the standard K\"ahler form. In particular, $3\b(\alpha)=\tfrac12\omega_{\text{FS}}\wedge\omega_{\text{FS}}$ is a $4$-form that modifies the curvature operator of $\C P^n$ to become positive-definite. The case of $\Hr P^n$ is analogous, in terms of its hyper-K\"ahler structure.
\end{remark}

Notice that the above argument does not apply to $\Ca P^2$, given that there are no submersions from round spheres to the Cayley plane.\footnote{Even more, for topological reasons, there are no fiber bundles $S^n\to\Ca P^2$, see \cite{browder}.} Actually, the following result (combined with Theorem~\ref{thm:submersions}) provides an alternative proof of this fact.

\begin{proposition}\label{prop:cross2}
The symmetric space $\Ca P^2$ does not have strongly positive curvature.
\end{proposition}

\begin{proof}
Assume by contradiction that $\Ca P^2$ has strongly positive curvature. Then, by Proposition~\ref{prop:ginvariance}, there exists an $\mathsf F_4$-invariant $\omega\in\Omega^4(\Ca P^2)$ such that $R+\omega$ is positive-definite. Notice that $\omega\neq0$, since $R$ is positive-semidefinite but has nontrivial kernel. Since $\Ca P^2=\mathsf F_4/\Spin(9)$ is a compact symmetric space, $\omega$ is $\mathsf F_4$-invariant if and only if it is harmonic \cite[p. 227]{helgason}. By Hodge theory, $[\omega]\in H^4(\Ca P^2,\R)$ is a nontrivial cohomology class, contradicting the fact that $b_4(\Ca P^2,\R)=0$.
\end{proof}

\begin{remark}\label{rem:zoltek}
The first examples of algebraic curvature operators $R\colon\Lambda^2 V\to\Lambda^2 V$, $\dim V\geq 5$, that have $\sec_R>0$ but 
do not have strongly positive curvature were found by Zoltek~\cite{zoltek}.
Note that any algebraic curvature operator can be realized as the curvature operator of a Riemannian manifold \emph{at one point} (see \cite[p. 104]{jaco}). However, to our knowledge, no closed manifolds with $\sec>0$ were known not to have strongly positive curvature. By the above, the Cayley plane $\Ca P^2$ is one such example. Moreover, the Berger space $B^{13}$ and certain Berger metrics on $S^{4n+3}$ provide further examples, see Remark~\ref{rem:b13} and Appendix~\ref{sec:berger}.
\end{remark}

\section{Strong Fatness and the Strong Wallach Theorem}
\label{sec:strongfatwallach}

All examples of closed homogeneous manifolds with $\sec>0$ (with one exception) are the total space of a homogeneous fibration (cf. Theorem~\ref{thm:homsp}). A unified treatment of these examples is given by the classic result of Wallach~\cite[Sec 7]{wa}, see also \cite{es} and \cite[Prop 4.3]{zillersurvey}. In this section, we strengthen it to handle strongly positive curvature.

Let $\H\subset\K\subset\G$ be compact Lie groups, and consider the \emph{homogeneous fibration}
\begin{equation}\label{eq:homfib}
\K/\H\longrightarrow \G/\H\stackrel{\pi}{\longrightarrow} \G/\K, \quad \pi(g\H)=g\K.
\end{equation}
Denote by $\h\subset \k\subset \mathfrak g$ the Lie algebras of $\H\subset\K\subset\G$. Fix a bi-invariant metric $Q$ on $\G$ and $Q$-orthogonal splittings
\begin{equation}\label{eq:mp}
\mathfrak g=\k\oplus\m, \quad [\k,\m]\subset\m, \quad \mbox{ and } \quad \k=\h\oplus\p, \quad [\h,\p]\subset\p,
\end{equation}
so that there are natural identifications of the tangent spaces
\begin{equation*}
\m\cong T_{(e\K)}\G/\K, \quad \p\cong T_{(e\H)}\K/\H \quad \mbox{ and } \quad \m\oplus\p\cong T_{(e\H)}\G/\H.
\end{equation*}
With the above, we also identify $\mathrm{Ad}$-invariant inner products on $\m$ with the induced $\G$-invariant metrics on $\G/\K$, $\mathrm{Ad}$-invariant elements of $\Lambda^k \m$ with the induced $\G$-invariant forms in $\Omega^k(\G/\K)$, and analogously for the homogeneous spaces $\K/\H$ and $\G/\H$. Consider the homogeneous metrics on $\G/\H$ given by
\begin{equation}\label{eq:gt}
\g_t=t\, Q|_{\p}+Q|_{\m}, \quad t>0,
\end{equation}
so that $\g_1$ is a normal homogeneous metric, and $\g_t$ is obtained by rescaling it by $t$ in the direction of the fibers of \eqref{eq:homfib}. %A familiar example of the above situation is given by the so-called \emph{Berger metrics} on spheres, see Appendix~\ref{sec:berger}.

\subsection{Strong Fatness}\label{subsec:strongfat}
Apart from positive curvature assumptions on the base and fiber, a crucial ingredient in Wallach's result is that \eqref{eq:homfib} is a \emph{fat bundle}. The concept of \emph{fatness}, introduced by Weinstein~\cite{weinstein2}, is related to positivity of the vertizontal planes, i.e., those spanned by a vertical and a horizontal vector, see also Ziller~\cite{fatness}. In the above situation, we can state it as:
\begin{equation*}
\text{(Fatness) For any } X\in\m \text{ and }Y\in\p, \text{ if }\|[X,Y]\|^2=0 \text{ then }X=0 \text{ or }Y=0.
\end{equation*}
%(Fatness) For any $X\in\m$ and $Y\in\p$, if $\|[X,Y]\|^2=0$ then $X=0$ or $Y=0$.
In order to establish the appropriate version of fatness for strongly positive curvature, consider the natural splittings
\begin{equation}\label{eq:lambdas}
\begin{aligned}
\Lambda^2(\m\oplus\p) &=\Lambda^2\m\oplus\Lambda^2\p\oplus(\m\otimes\p), \text { and }\\
\Lambda^4(\m\oplus\p) &=\Lambda^4\m\oplus\Lambda^4\p\oplus(\Lambda^3\m\otimes\p)\oplus(\Lambda^2\m\otimes\Lambda^2\p)\oplus (\m\otimes\Lambda^3\p).
\end{aligned}
\end{equation}
Consider the linear map $L$ given on decomposable elements of $\m\otimes\p$ by
\begin{equation}\label{eq:L}
L\colon\m\otimes\p\longrightarrow\m, \quad L(X\wedge Y)=[X,Y],
\end{equation}
and extended by linearity to the entire $\m\otimes\p$. This linear map induces the operator
\begin{equation}\label{eq:F}
F\colon \m\otimes\p\longrightarrow\m\otimes\p, \quad F=L^*L,
\end{equation}
which is clearly positive-semidefinite and has nontrivial kernel equal to $\ker L$. Given that $F$ is a self-adjoint linear operator on a subspace of $\Lambda^2(\m\oplus\p)$, it makes sense to add to it a $4$-form $\tau\in\Lambda^2\m\otimes\Lambda^2\p\subset\Lambda^4(\m\oplus\p)$, by using \eqref{eq:wrw}. The appropriate strengthening of fatness needed for strongly positive curvature is:
\begin{align*}
&\text{(Strong Fatness) There exists } \tau \in\Lambda^2\m\otimes\Lambda^2\p\subset\Lambda^4(\m\oplus\p) \text{ such that } \\
&\text{the operator } (F+\tau)\colon \m\otimes\p\to\m\otimes\p \text{ is positive-definite.}
\end{align*}
Clearly, a decomposable element of $\m\otimes\p$ is of the form $X\wedge Y$, where $X\in\m$ and $Y\in\p$. Since $\big\langle(F+\tau)(X\wedge Y),X\wedge Y\big\rangle=\|L(X\wedge Y)\|^2=\|[X,Y]\|^2$,
strong fatness automatically implies fatness.

\subsection{First-order Lemma}
We now state the following elementary \emph{first-order} perturbation argument, that is used throughout the rest of the paper.

\begin{lemma}\label{lemma:firstorder}
Let $V$ be a finite-dimensional vector space with an inner product $\langle\cdot,\cdot\rangle$. Let $A$ and $B$ be self-adjoint operators on $V$, such that $A$ is positive-semidefinite and $B\colon\ker A\to\ker A$ is positive-definite. Then $A+\varepsilon B$ is positive-definite for all $\varepsilon>0$ sufficiently small.
\end{lemma}
%\begin{proof}
%Let $S_V=\{x\in V:\|x\|=1\}$ be the unit sphere in $V$, and set
%\begin{equation*}
%f\colon [-1,1]\times S_V\longrightarrow\R, \quad f(t,x)=\langle (A+tB)x,x\rangle.
%\end{equation*}
%Then $f(0,x)\geq0$, and $\frac{\dd}{\dd t}f(t,x)\big|_{t=0}>0$ for all $x$ such that $f(0,x)=0$. It follows from the Taylor polynomial of $f(t,x)$ at $t=0$ that, since $S_V$ is compact, $f(\varepsilon,x)>0$ for all $x\in S_V$ and $\varepsilon>0$ sufficiently small, cf. \cite[Prop 3.3]{bettiol}. Thus, $A+\varepsilon B$ is positive-definite for all $\varepsilon>0$ sufficiently small.
%\end{proof}

\subsection{Strong Wallach Theorem}
We are now ready for the strongly positive curvature version of the result of Wallach~\cite[Sec 7]{wa} on homogeneous fibrations with $\sec>0$, see also \cite{es} and \cite[Prop 4.3]{zillersurvey}.

\begin{theorem}\label{thm:wallach}
Suppose that the homogeneous fibration \eqref{eq:homfib} satisfies:
\begin{itemize}
\item[(i)] the base $(\G/\K,Q|_\m)$ is a CROSS different from $\Ca P^2$ and $(\mathfrak g,\k)$ is a symmetric pair;
\item[(ii)] the fiber $(\K/\H,Q|_\p)$ has constant positive curvature, and either $(\k,\h)$ is a symmetric pair or $\dim\K/\H\leq3$;
\item[(iii)] strong fatness (see Subsection~\ref{subsec:strongfat}).
\end{itemize}
Then $(\G/\H,\g_t)$ has strongly positive curvature for all $0<t<1$.
\end{theorem}

\begin{proof}
The starting point to show that $(\G/\H,\g_t)$ has strongly positive curvature is to compute its curvature operator. This is done with the Riemannian submersions
\begin{equation}\label{eq:setupwallach}
\left(\G\times\K,Q+\tfrac1sQ|_\k\right)\stackrel{\pi_1}{\longrightarrow} (\G,Q_t)\stackrel{\pi_2}{\longrightarrow} (\G/\H,\g_t),
\end{equation}
where $\pi_2$ is the quotient map and $\pi_1$ is of the form \eqref{eq:submcheeger}, corresponding to the fact that $Q_t=t\, Q|_{\k}+Q|_{\m}$, $t<1$, is the result of a Cheeger deformation of $(\G,Q)$ with respect to the $\K$-action by left multiplication. From \eqref{eq:ptct}, it is easy to see that $t=\frac{1}{1+s}$, where $s$ is the parameter of the Cheeger deformation. In particular, large values of $s>0$ correspond to small values of $0<t<1$.

Denote by $\alpha_1$ and $\alpha_2$ the positive-semidefinite operators on $\Lambda^2(\m\oplus\p)$ induced as in \eqref{eq:alpha} by the $A$-tensor of the Riemannian submersions $\pi_1$ and $\pi_2$ respectively. A direct computation shows that, for $X,Y,Z,W\in\m\oplus\p$,
\begin{equation}
\begin{aligned}
\langle\alpha_1(\overline X\wedge\overline  Y),\overline Z\wedge\overline  W\rangle &=\tfrac{1-t}{4}Q\big([X_\m,Y_\m]+t[X_\p,Y_\p],[Z_\m,W_\m]+t[Z_\p,W_\p]\big),\\
\langle\alpha_2(X\wedge Y),Z\wedge W\rangle &=\tfrac{t}{4}Q\big([X,Y]_\h,[Z,W]_\h\big),
\end{aligned}
\end{equation}
where $V_\m$, $V_\p$ and $V_\h$ respectively denote the components of $V\in\mathfrak g$ in $\m$, $\p$ and $\h$. Thus, applying twice formula \eqref{eq:oneill} with the setup \eqref{eq:setupwallach}, we get that the curvature operator $R_t\colon\Lambda^2(\m\oplus\p)\to\Lambda^2(\m\oplus\p)$ of $(\G/\H,\g_t)$ is given by:\footnote{Note that when $t=1$, this is the curvature \eqref{eq:curvhomsp} of the normal homogeneous space $(\G/\H,\g_1)$.}
\begin{align}\label{eq:curvopwallach}
\langle R_t(X\wedge Y),Z\wedge W\rangle_t
%&=\langle R_{\,\G\times\K}(\overline X\wedge\overline  Y),\overline Z\wedge\overline  W\rangle\\
%&\quad +3\langle\alpha_1(\overline X\wedge\overline  Y),\overline Z\wedge\overline  W\rangle +3\langle\alpha_2(\overline X\wedge\overline  Y),\overline Z\wedge\overline  W\rangle \\
%&\quad -3\b(\alpha_1)(\overline X,\overline Y,\overline Z,\overline W)-3\b(\alpha_2)(\overline X,\overline Y,\overline Z,\overline W)\\
&=\tfrac14 Q\big([X_\m+tX_\p,Y_\m+tY_\p],[Z_\m+tZ_\p,W_\m+tW_\p]\big)\nonumber\\
&\quad +\tfrac{t(1-t)^3}{4}Q\big([X_\p,Y_\p],[Z_\p,W_\p]\big)\nonumber\\
&\quad +\tfrac{3(1-t)}{4}Q\big([X_\m,Y_\m]+t[X_\p,Y_\p],[Z_\m,W_\m]+t[Z_\p,W_\p]\big)\\
&\quad +\tfrac{3t}{4}Q\big([X,Y]_\h,[Z,W]_\h\big)\nonumber\\
&\quad -3\b(\alpha_1)(\overline X,\overline Y,\overline Z,\overline W)-3\b(\alpha_2)(X,Y,Z,W).\nonumber
\end{align}
Fix $0<t<1$ and consider the positive-semidefinite operator given by
\begin{equation*}
\widehat R:=R_t+3\b(\alpha_1)+3\b(\alpha_2)\colon \Lambda^2(\m\oplus\p)\longrightarrow\Lambda^2(\m\oplus\p).
\end{equation*}
Then $(\G/\H,\g_t)$ has strongly positive curvature if and only if there exists $\omega\in\Lambda^4(\m\oplus\p)$ such that $\widehat R+\omega$ is positive-definite.

Using the natural decomposition \eqref{eq:lambdas}, we can write $\widehat R$ in blocks as follows:
\begin{equation}\label{eq:rhat}
\bordermatrix{&\Lambda^2\m & &\Lambda^2\p & & \m\otimes\p \cr
               \Lambda^2\m  &\rule[-1.2ex]{0pt}{0pt} \rule{0pt}{2.5ex} \widehat R_{11} & & \widehat R_{12} & & 0\cr
                \Lambda^2\p & \rule[-1.2ex]{0pt}{0pt} \rule{0pt}{2.5ex} \widehat R_{12}^{\,\mathrm{t}} & & \widehat R_{22} & & 0\cr
                \m\otimes\p & 0  \rule[-1.2ex]{0pt}{0pt} \rule{0pt}{2.5ex} & &   0    &   & \widehat R_{33}}
\end{equation}
The zeros above are obtained directly from \eqref{eq:mp} and \eqref{eq:curvopwallach}, using that $[\m,\m]\subset\k$, since $\G/\K$ is a CROSS. Moreover, the self-adjoint positive-semidefinite operators
\begin{equation*}
\widehat R_{11}\colon\Lambda^2\m\longrightarrow\Lambda^2\m,\quad \widehat R_{22}\colon\Lambda^2\p\longrightarrow\Lambda^2\p,\quad \widehat R_{33}\colon\m\otimes\p\longrightarrow\m\otimes\p
\end{equation*}
can be explicitly computed from \eqref{eq:curvopwallach} as follows:
\begin{align*}
\big\langle\widehat R_{11}(X\wedge Y),Z\wedge W\big\rangle_t&=(1-\tfrac{3t}{4})Q\big([X,Y],[Z,W]\big)+\tfrac{3t}{4}Q\big([X,Y]_\h,[Z,W]_\h\big),\\
\big\langle\widehat R_{22}(X\wedge Y),Z\wedge W\big\rangle_t&=\tfrac{t}{4}Q\big([X,Y],[Z,W]\big)+\tfrac{3t}{4}Q\big([X,Y]_\h,[Z,W]_\h\big),\\
\big\langle\widehat R_{33}(X\wedge Y),Z\wedge W\big\rangle_t&=\tfrac{t^2}{4}Q\big([X,Y],[Z,W]\big).
\end{align*}
We now use the hypotheses on $\G/\K$ and $\K/\H$ to relate their curvature operators, given by \eqref{eq:alphahomsp} and \eqref{eq:curvhomsp}, with the above. Since $[\m,\m]\subset\k$, we have that $\alpha_{\G/\K}$ is a multiple of $R_{\G/\K}$ and hence $\b(\alpha_{\G/\K})=0$. Analogously, if $(\k,\h)$ is a symmetric pair, then $[\p,\p]\subset\h$ and hence $\alpha_{\K/\H}$ is a multiple of $R_{\K/\H}$, so $\b(\alpha_{\K/\H})=0$. Else, $\dim \K/\H\leq3$ and hence $\b(\alpha_{\K/\H})\in \Lambda^4\p=\{0\}$. In either case,
\begin{equation*}
\b(\alpha_{\G/\K})=0\quad \text{ and } \quad\b(\alpha_{\K/\H})=0.
\end{equation*}
Consequently, we get the following:
\begin{align*}
\big\langle\widehat R_{11}(X\wedge Y),Z\wedge W\big\rangle_t&=(1-\tfrac{3t}{4})\big\langle R_{\G/\K}(X\wedge Y),Z\wedge W\big\rangle+\tfrac{3t}{4}Q\big([X,Y]_\h,[Z,W]_\h\big),\\
\big\langle\widehat R_{22}(X\wedge Y),Z\wedge W\big\rangle_t&=t\,\big\langle R_{\K/\H}(X\wedge Y),Z\wedge W\big\rangle,\\
\big\langle\widehat R_{33}(X\wedge Y),Z\wedge W\big\rangle_t&=\tfrac{t^2}{4}\big\langle F(X\wedge Y),Z\wedge W\big\rangle,
\end{align*}
where $F$ is related to strong fatness, and given by \eqref{eq:F}.

Let us first analyze the restriction of $\widehat R$ to $\Lambda^2\m\oplus\Lambda^2\p$, i.e., the upper $2\times 2$ block of \eqref{eq:rhat}. Since $\K/\H$ has constant positive curvature, we get that $\widehat R_{22}=R_{\K/\H}$ is positive-definite. Thus, the kernel of the positive-semidefinite operator
\begin{equation}\label{eq:restrictedR}
\widehat R\colon\Lambda^2\m\oplus\Lambda^2\p\longrightarrow\Lambda^2\m\oplus\Lambda^2\p
\end{equation}
must be contained in $\Lambda^2\m$. It follows from Proposition~\ref{prop:cross1} that, since $\G/\K$ is a CROSS different from $\Ca P^2$, there exists $\eta\in\Lambda^4\m$ such that $R_{\G/\K}+\eta$, and hence $\widehat R_{11}+\eta$, is positive-definite. In particular, we have that
\begin{equation*}
\bordermatrix{&\Lambda^2\m & &\Lambda^2\p & & \m\otimes\p \cr
               \Lambda^2\m  & \rule[-1.2ex]{0pt}{0pt} \rule{0pt}{2.5ex} \eta & & 0 & & 0\cr
                \Lambda^2\p & \rule[-1.2ex]{0pt}{0pt} \rule{0pt}{2.5ex} 0& & 0 & & 0\cr
                \m\otimes\p   & \rule[-1.2ex]{0pt}{0pt} \rule{0pt}{2.5ex} 0 & & 0 & & 0}
\end{equation*}
is positive-definite on the kernel of \eqref{eq:restrictedR}. Thus, Lemma~\ref{lemma:firstorder} gives an $\varepsilon_1>0$ such that $\widehat R+\varepsilon_1\eta$ is positive-definite on $\Lambda^2\m\oplus\Lambda^2\p$.

Finally, we analyze the positive-semidefinite operator 
\begin{equation}\label{eq:modifiedR}
\big(\widehat R+\varepsilon_1\eta\big)\colon \Lambda^2(\m\oplus\p)\longrightarrow \Lambda^2(\m\oplus\p).
\end{equation}
Since its restriction to $\Lambda^2\m\oplus\Lambda^2\p$ is positive-definite, its kernel must lie in $\m\otimes\p$. The restriction of \eqref{eq:modifiedR} to this subspace coincides with $\widehat R_{33}=\tfrac{t^2}{4}F$. By strong fatness, there exists $\tau\in\Lambda^2\m\otimes\Lambda^2\p$ such that $F+\tau$ is positive-definite on $\m\otimes\p$. In particular, we have that
\begin{equation*}
\bordermatrix{&\Lambda^2\m & &\Lambda^2\p & & \m\otimes\p \cr
               \Lambda^2\m  & \rule[-1.2ex]{0pt}{0pt} \rule{0pt}{2.5ex} 0 & & \tau & & 0\cr
                \Lambda^2\p & \rule[-1.2ex]{0pt}{0pt} \rule{0pt}{2.5ex} \tau & & 0 & & 0\cr
                \m\otimes\p   & \rule[-1.2ex]{0pt}{0pt} \rule{0pt}{2.5ex} 0 & & 0 & & \tau}
\end{equation*}
is positive-definite on the kernel of \eqref{eq:modifiedR}. Thus, Lemma~\ref{lemma:firstorder} gives an $\varepsilon_2>0$ such that $\widehat R+\varepsilon_1\eta+\varepsilon_2\tau$ is positive-definite on $\Lambda^2(\m\oplus\p)$. In other words, $\omega=\varepsilon_1\eta+\varepsilon_2\tau\in\Lambda^4(\m\oplus\p)$ is such that the modified curvature operator $\widehat R+\omega$ is positive-definite, hence $(\G/\H,\g_t)$ has strongly positive curvature.
\end{proof}

\section{Homogeneous spaces with strongly positive curvature}
\label{sec:homsp}

Homogeneous spaces that admit an invariant metric with $\sec>0$ were classified in even dimensions by Wallach~\cite{wa} and in odd dimensions by B\'erard-Bergery~\cite{bb}, see also \cite{aw,berger}. Apart from the CROSS described in Section~\ref{sec:cross}, other examples only appear in dimensions $6$, $7$, $12$, $13$ and $24$. The complete list is given as follows, see \cite{grove-survey,zillersurvey}.

\begin{theorem}\label{thm:homsp}
Apart from the CROSS, the only simply-connected closed manifolds to admit a homogeneous metric with $\sec>0$ are:
\begin{itemize}
\item[(i)] Wallach flag manifolds: $W^6=\SU(3)/\mathsf T^2$, $W^{12}=\Sp(3)/\Sp(1)\Sp(1)\Sp(1)$ and $W^{24}=\mathrm F_4/\Spin(8)$;
\item[(ii)] Aloff-Wallach spaces: $W^7_{k,\ell}=\SU(3)/\mathsf S^1_{k,\ell}$, $\gcd(k,\ell)=1$, $k\ell (k+\ell)\neq0$;
\item[(iii)] Aloff-Wallach space: $W_{1,1}^7=\SU(3)\SO(3)/\U(2)$;
\item[(iv)] Berger spaces: $B^7=\SO(5)/\SO(3)$ and $B^{13}=\SU(5)/\Sp(2)\cdot\mathsf S^1$.
\end{itemize}
\end{theorem}

In this section, we study which of the above homogeneous spaces also have strongly positive curvature. All of the above can be shown to have an invariant metric with $\sec>0$ by using Wallach's theorem on homogeneous fibrations, with the exception of the Berger space $B^7$, see \cite[Sec 4]{zillersurvey}. Our strong version of Wallach's theorem (Theorem~\ref{thm:wallach}) has slightly more restrictive hypotheses, preventing it from being applicable to the Wallach flag manifold $W^{24}$, which fibers over $\Ca P^2$. Apart from $W^{24}$, all the remaining examples of homogeneous fibrations can be handled with Theorem~\ref{thm:wallach}. Finally, the exceptional case of $B^7$ is dealt with separately, by directly computing its modified curvature operator. Altogether, we prove:

\begin{theorem}\label{thm:strongposhomsp}
All simply-connected closed homogeneous spaces that admit a homogeneous metric with $\sec>0$ also admit a homogeneous metric with strongly positive curvature, except for $\Ca P^2$ and possibly $W^{24}$.
\end{theorem}

\subsection{\texorpdfstring{Flag manifold $W^6$}{Flag manifold W6}}
The Wallach flag manifold $W^6$ is the total space of a homogeneous fibration \eqref{eq:homfib}, where the Lie groups $\H\subset\K\subset\G$ are 
\begin{equation*}
\mathsf T^2 \subset \U(2)\subset\SU(3).
\end{equation*}
The inclusion $\U(2)\subset\SU(3)$ is given by $\U(2)=\{\diag(A,\det\overline{A})\in\SU(3):A\in\U(2)\}$, and $\mathsf T^2\subset\U(2)$ is the maximal torus formed by diagonal matrices. We thus identify the corresponding homogeneous fibration as
\begin{equation*}
\C P^1\longrightarrow W^6\longrightarrow \C P^2.
\end{equation*}
The base $\C P^2$ is a CROSS and the fiber $\C P^1\cong S^2(\tfrac12)$ has constant positive curvature and dimension $\leq3$. The only remaining hypothesis to apply Theorem~\ref{thm:wallach} is strong fatness, for which we must analyze the appropriate Lie brackets.

The $Q$-orthogonal complement of $\h=\mathfrak t^2$ in $\mathfrak g=\mathfrak{su}(3)$ is parametrized by:
\begin{equation*}
\C^2\oplus \C\ni (z,w)\longmapsto\left(\begin{matrix}
                    0 & w & z_1 \\
                    -\overline{w} &  0 & z_2 \\
                    -\overline{z_1} & -\overline{z_2} & 0
                  \end{matrix}\right)\in \m\oplus\p\subset\mathfrak{su}(3).
\end{equation*}
Let $\{e_i\}$ be a $Q$-orthonormal basis of $\mathfrak{su}(2)$ such that $\{e_1,e_2,e_3,e_4\}$ and $\{e_5,e_6\}$ correspond respectively to the standard basis of $\m\cong\C^2$ and $\p\cong\C$. The linear maps $L$ and hence $F=L^*L$, see \eqref{eq:L} and \eqref{eq:F}, are determined by the following table of Lie brackets:
\begin{equation*}
\begin{array}{c | r r}
[\cdot,\cdot]& e_5\rule[-1.2ex]{0pt}{0pt} \rule{0pt}{2.5ex}  & e_6 \\
\hline
e_1\rule[-1.2ex]{0pt}{0pt} \rule{0pt}{2.5ex}  & e_3 & -e_4 \\
e_2\rule[-1.2ex]{0pt}{0pt} \rule{0pt}{2.5ex}  & e_4 & e_3  \\
e_3\rule[-1.2ex]{0pt}{0pt} \rule{0pt}{2.5ex}  & -e_1 & -e_2 \\
e_4\rule[-1.2ex]{0pt}{0pt} \rule{0pt}{2.5ex}  & -e_2 & e_1  \\
\end{array}
\end{equation*}
This implies that $\ker F=\ker L$ is spanned by the following $4$ vectors of $\m\otimes\p$:
\begin{align*}
e_3\wedge e_5&+e_4\wedge e_6, & -e_3\wedge e_6&+e_4\wedge e_5, & -e_1\wedge e_5&+ e_2\wedge e_6, & e_1\wedge e_6&+e_2\wedge e_5.
\end{align*}
Consider the self-adjoint operator induced by the $\H$-invariant $4$-form $\tau\in\Lambda^2\m\otimes\Lambda^2\p$,
\begin{equation*}
\tau=(e_1\wedge e_2-e_3\wedge e_4)\otimes(e_5\wedge e_6).
\end{equation*}
The restriction $\tau\colon\ker F\to\ker F$ is the identity operator, hence positive-definite. From Lemma~\ref{lemma:firstorder}, there exists $\varepsilon>0$ such that $(F+\varepsilon\tau)\colon \m\otimes\p\to\m\otimes\p$ is positive-definite, proving strong fatness. Thus, by Theorem~\ref{thm:wallach}, the homogeneous space $(W^6,\g_t)$ has strongly positive curvature for all $0<t<1$.

\subsection{\texorpdfstring{Flag manifold $W^{12}$}{Flag manifold W12}}\label{sec:w12}
The Wallach flag manifold $W^{12}$ is the total space of a homogeneous fibration \eqref{eq:homfib}, where the Lie groups $\H\subset\K\subset\G$ are
\begin{equation*}
\Sp(1)\Sp(1)\Sp(1) \subset \Sp(2)\Sp(1)\subset\Sp(3).
\end{equation*}
The inclusions are given by $\Sp(2)\Sp(1)=\{\diag(A,q):A\in\Sp(2),q\in\Sp(1)\}$, and $\Sp(1)\Sp(1)\Sp(1)\subset\Sp(2)\Sp(1)=\{\diag(q_1,q_2,q_3):q_j\in\Sp(1)\}$. We thus identify the corresponding homogeneous fibration as
\begin{equation*}
\Hr P^1\longrightarrow W^{12}\longrightarrow \Hr P^2.
\end{equation*}
The base $\Hr P^2$ is a CROSS, the fiber $\Hr P^1\cong S^4(\tfrac12)$ has constant positive curvature, and $(\k,\h)$ is a symmetric pair. The only remaining hypothesis to apply Theorem~\ref{thm:wallach} is strong fatness, for which we must analyze the appropriate Lie brackets.

The $Q$-orthogonal complement of $\h=\mathfrak{sp}(1)\oplus\mathfrak{sp}(1)\oplus\mathfrak{sp}(1)$ in $\mathfrak g=\mathfrak{sp}(3)$ is:
\begin{equation*}
\Hr^2\oplus \Hr\ni (z,w)\longmapsto\left(\begin{matrix}
                    0 & w & z_1 \\
                    -\overline{w} &  0 & z_2 \\
                    -\overline{z_1} & -\overline{z_2} & 0
                  \end{matrix}\right)\in \m\oplus\p\subset\mathfrak{sp}(3).
\end{equation*}
Let $\{e_i\}$ be a $Q$-orthonormal basis of $\mathfrak{sp}(3)$ such that $\{e_1,\dots,e_8\}$ and $\{e_9,\dots,e_{12}\}$ correspond respectively to the standard basis of $\m\cong\Hr^2$ and $\p\cong\Hr$. The linear maps $L$ and hence $F=L^*L$, see \eqref{eq:L} and \eqref{eq:F}, are determined by the following table of Lie brackets:
\begin{equation*}
\begin{array}{c | r r r r}
[\cdot,\cdot]& e_9\rule[-1.2ex]{0pt}{0pt} \rule{0pt}{2.5ex}  & e_{10} & e_{11} & e_{12} \\
\hline
e_1\rule[-1.2ex]{0pt}{0pt} \rule{0pt}{2.5ex}  & e_5 & -e_6  & -e_7 & -e_8 \\
e_2\rule[-1.2ex]{0pt}{0pt} \rule{0pt}{2.5ex}  & e_6 & e_5  & e_8 & -e_7 \\
e_3\rule[-1.2ex]{0pt}{0pt} \rule{0pt}{2.5ex}  & e_7 & -e_8  & e_5 & e_6 \\
e_4\rule[-1.2ex]{0pt}{0pt} \rule{0pt}{2.5ex}  & e_8 & e_7  & -e_6 & e_5 \\
e_5\rule[-1.2ex]{0pt}{0pt} \rule{0pt}{2.5ex}  & -e_1 & -e_2  & -e_3 & -e_4 \\
e_6\rule[-1.2ex]{0pt}{0pt} \rule{0pt}{2.5ex}  & -e_2 & e_1  & e_4 & -e_3 \\
e_7\rule[-1.2ex]{0pt}{0pt} \rule{0pt}{2.5ex}  & -e_3 & -e_4  & e_1 & e_2 \\
e_8\rule[-1.2ex]{0pt}{0pt} \rule{0pt}{2.5ex}  & -e_4 & e_3  & -e_2 & e_1 \\
\end{array}
\end{equation*}
This implies that $\ker F=\ker L$ is spanned by the following $24$ vectors of $\m\otimes\p$:
\begin{align*}
e_5\wedge e_9&+e_8\wedge e_{12}, & -e_5\wedge e_{10}&+e_8\wedge e_{11}, & e_5\wedge e_{11}&+ e_8\wedge e_{10}, \\
-e_5\wedge e_{12}&+e_8\wedge e_9, & e_5\wedge e_{10}&+e_7\wedge e_{12}, & e_5\wedge e_9&+e_7\wedge e_{11}, \\
-e_5\wedge e_{12}&+ e_7\wedge e_{10}, & -e_5\wedge e_{11}&+e_7\wedge e_9, & -e_5\wedge e_{11}&+e_6\wedge e_{12}, \\
e_5\wedge e_{12}&+e_6\wedge e_{11}, & e_5\wedge e_9&+ e_6\wedge e_{10}, & -e_5\wedge e_{10}&+e_6\wedge e_9,\\
-e_1\wedge e_9&+e_4\wedge e_{12}, & -e_1\wedge e_{10}&+e_4\wedge e_{11}, & e_1\wedge e_{11}&+ e_4\wedge e_{10}, \\
e_1\wedge e_{12}&+e_4\wedge e_9, & e_1\wedge e_{10}&+e_3\wedge e_{12}, & -e_1\wedge e_9&+e_3\wedge e_{11}, \\
-e_1\wedge e_{12}&+ e_3\wedge e_{10}, & e_1\wedge e_{11}&+e_3\wedge e_9, & -e_1\wedge e_{11}&+e_2\wedge e_{12}, \\
e_1\wedge e_{12}&+e_2\wedge e_{11}, & -e_1\wedge e_9&+ e_2\wedge e_{10}, & e_1\wedge e_{10}&+e_2\wedge e_9.
\end{align*}
Consider the self-adjoint operator induced by the $\H$-invariant $4$-form $\tau\in\Lambda^2\m\otimes\Lambda^2\p$,
\begin{align*}
\tau&=(e_5\wedge e_6+e_7\wedge e_8)\otimes(e_{11}\wedge e_{12}-e_9\wedge e_{10}) \\
&\quad+(e_6\wedge e_8-e_5\wedge e_7)\otimes(e_{9}\wedge e_{11}+e_{10}\wedge e_{12})\\
&\quad +(e_5\wedge e_8+e_6\wedge e_7)\otimes(e_{10}\wedge e_{11}-e_{9}\wedge e_{12})\\
&\quad+(e_1\wedge e_2+e_3\wedge e_4)\otimes(e_9\wedge e_{10}+e_{11}\wedge e_{12})\\
&\quad +(e_1\wedge e_3-e_2\wedge e_4)\otimes(e_9\wedge e_{11}-e_{10}\wedge e_{12})\\
&\quad+(e_1\wedge e_4+e_2\wedge e_3)\otimes(e_9\wedge e_{12}+e_{10}\wedge e_{11}).
\end{align*}
The restriction $\tau\colon\ker F\to\ker F$ is the identity operator, hence positive-definite. From Lemma~\ref{lemma:firstorder}, there exists $\varepsilon>0$ such that $(F+\varepsilon\tau)\colon \m\otimes\p\to\m\otimes\p$ is positive-definite, proving strong fatness. Thus, by Theorem~\ref{thm:wallach}, the homogeneous space $(W^{24},\g_t)$ has strongly positive curvature for all $0<t<1$.

\subsection{\texorpdfstring{Aloff-Wallach spaces $W^7_{k,\ell}$}{Aloff-Wallach spaces W7kl}}
Consider the Lie groups $\H\subset\K\subset\G$ given by
\begin{equation*}
\mathsf S^1_{k,\ell} \subset \U(2)\subset\SU(3).
\end{equation*}
The inclusion $\U(2)\subset\SU(3)$ is the same as in the flag manifold $W^6$, and $\mathsf S^1_{k,\ell}=\{\diag(z^k,z^\ell,\overline{z}^{k+\ell}):z\in S^1\}$ is a circle with slope $(k,\ell)$ inside the maximal torus of $\U(2)$. Up to the appropriate equivalences, the nontrivial cases correspond to $0<k\leq\ell$ and $\gcd(k,\ell)=1$. The above groups induce the homogeneous fibration
\begin{equation*}
S^3/\Z_{k+\ell}\longrightarrow W^7_{k,\ell}\longrightarrow \C P^2.
\end{equation*}
The base $\C P^2$ is a CROSS and the fiber $S^3/\Z_{k+\ell}$ has constant positive curvature and dimension $\leq3$. The only remaining hypothesis to apply Theorem~\ref{thm:wallach} is strong fatness, for which we must analyze the appropriate Lie brackets.

For convenience of notation, set $r\in(0,1]$ and $s\in (1,3]$ to be the numbers
\begin{equation*}
r:=k/\ell \quad \text{ and } \quad s:=1+r+r^2.
\end{equation*}
The $Q$-orthogonal complement of the Lie algebra of $\mathsf S^1_{k,\ell}$ in $\mathfrak{su}(3)$ is:
\begin{equation*}
\C^2\oplus (\R\oplus\C)\ni (z,x,w)\longmapsto\left(\begin{matrix}
                    \frac{(2+r)i}{\sqrt{3s}} & w & z_1 \\
                    -\overline{w} &  -\frac{(2r+1)i}{\sqrt{3s}} & z_2 \\
                    -\overline{z_1} & -\overline{z_2} & \frac{(r-1)i}{\sqrt{3s}}
                  \end{matrix}\right)\in \m\oplus\p\subset\mathfrak{su}(3).
\end{equation*}
Let $\{e_i\}$ be a $Q$-orthonormal basis of $\mathfrak{su}(3)$ such that $\{e_1,e_2,e_3,e_4\}$ and $\{e_5,e_6,e_7\}$ correspond respectively to the standard basis of $\m\cong\C^2$ and $\p\cong\R\oplus\C$. The linear maps $L$ and hence $F=L^*L$, see \eqref{eq:L} and \eqref{eq:F}, are determined by the following table of Lie brackets:
\begin{equation*}
\begin{array}{c | r r r}
[\cdot,\cdot]& e_5\rule[-1.2ex]{0pt}{0pt} \rule{0pt}{2.5ex}  & e_6 & e_7 \\
\hline
e_1\rule[-1.2ex]{0pt}{0pt} \rule{0pt}{2.5ex}  & -\sqrt\frac{3}{s}e_2 & e_3 & -e_4 \\
e_2\rule[-1.2ex]{0pt}{0pt} \rule{0pt}{2.5ex}  & \sqrt\frac{3}{s}e_1 & e_4 & e_3  \\
e_3\rule[-1.2ex]{0pt}{0pt} \rule{0pt}{2.5ex}  & r\sqrt\frac{3}{s}e_4& -e_1 & -e_2 \\
e_4\rule[-1.2ex]{0pt}{0pt} \rule{0pt}{2.5ex}  & -r\sqrt\frac{3}{s}e_3 & -e_2 & e_1  \\
\end{array}
\end{equation*}
This implies that $\ker F=\ker L$ is spanned by the following $8$ vectors of $\m\otimes\p$:
\begin{align*}
-\sqrt{\tfrac{s}{3}}e_2\wedge e_5&+e_4\wedge e_7, & -\sqrt{\tfrac{s}{3}}e_1\wedge e_5&+e_4\wedge e_6, & r\sqrt{\tfrac{3}{s}}e_1\wedge e_6&+ e_4\wedge e_5,\\
-\sqrt{\tfrac{s}{3}}e_1\wedge e_5&+e_3\wedge e_7 & \sqrt{\tfrac{s}{3}}e_2\wedge e_5&+e_3\wedge e_6, & r\sqrt{\tfrac{3}{s}}e_1\wedge e_7&+e_3\wedge e_5,\\
-e_1\wedge e_6&+e_2\wedge e_7, & e_1\wedge e_7&+e_2\wedge e_6.&&
\end{align*}
Consider the operator induced by the $\H$-invariant $4$-form $\tau_{a,b}\in\Lambda^2\m\otimes\Lambda^2\p$,
\begin{align*}
\tau_{a,b}&=(a\,e_1\wedge e_2+b\,e_3\wedge e_4)\otimes(e_6\wedge e_7)+\sqrt3(e_1\wedge e_3+e_2\wedge e_4)\otimes(e_5\wedge e_7)\\
&\quad +\sqrt3(e_1\wedge e_4-e_2\wedge e_3)\otimes(e_5\wedge e_6)
\end{align*}%a=a[4]>0 small, b=a[1]<0 small
The restriction $\tau_{a,b}\colon\ker F\to\ker F$ is positive-definite if and only if
\begin{equation*}
a>0, \quad a-\frac{r}{4\sqrt{s}}a^2>0, \quad b+\frac{1}{4\sqrt{s}}b^2<0.
\end{equation*}
Since $r>0$ and $s>0$, there exist $a>0$ and $b<0$ sufficiently small such that $\tau_{a,b}\colon\ker F\to\ker F$ is positive-definite, proving strong fatness. Thus, by Theorem~\ref{thm:wallach}, all the homogeneous spaces $(W^{7}_{k,\ell},\g_t)$, $k\ell(k+\ell)\neq0$, have strongly positive curvature for all $0<t<1$.

\subsection{\texorpdfstring{Berger space $B^7$}{Berger space B7}}\label{subsec:b7}
Differently from all the previous examples, the Berger space $B^7=\SO(5)/\SO(3)$ does not admit a homogeneous fibration. The inclusion $\SO(3)\subset\SO(5)$ comes from the conjugation action of $\SO(3)$ on the space of symmetric traceless $3\times3$ matrices, which is identified with $\R^5$. Let $Q(X,Y)=-\frac{1}{10}\Re\operatorname{tr}(XY)$ be the bi-invariant metric on $\mathfrak{so}(5)$ and $\mathfrak{so}(5)=\mathfrak{so}(3)\oplus\m$ be a $Q$-orthogonal splitting. The isotropy action of $\SO(3)$ on $\m$ is irreducible, hence there is a unique $\SO(5)$-invariant metric, up to homotheties, which is known to have $\sec>0$. Denote by $\g$ this normal homogeneous metric corresponding to $Q|_\m$.

In order to prove that $(B^7,\g)$ has strongly positive curvature, we explicitly compute the positive-semidefinite operator $\widehat R=R_{\G/\H}+3\b(\alpha_{\G/\H})$, given by
%\begin{equation*}
%\big\langle \widehat R(X\wedge Y),Z\wedge W\big\rangle =\big\langle R_{\G/\H}(X\wedge Y),Z\wedge W\big\rangle+3\b(\alpha_{\G/\H})(X,Y,Z,W),
%\end{equation*}
\begin{equation}\label{eq:rhatb7}
\big\langle \widehat R(X\wedge Y),Z\wedge W\big\rangle =\tfrac14Q([X,Y],[Z,W])+\tfrac34Q\big([X,Y]_\h,[Z,W]_\h\big),
\end{equation}
see Example~\ref{ex:homsp}. It is clear that $(B^7,\g)$ has strongly positive curvature if and only if there exists $\omega\in\Lambda^4\m$ such that $\widehat R+\omega$ is positive-definite. The subspace $\m\subset\mathfrak{so}(5)$ can be parametrized as follows.
\begin{equation*}
\m=\left\{
\left(\begin{smallmatrix}
                    0\rule[-1.2ex]{0pt}{0pt} \rule{0pt}{2.5ex} & \sqrt{5}x_7 & \sqrt{2}x_1 & \sqrt{5}x_6 & \sqrt{2}x_4 \\
                  -\sqrt{5}x_7\rule[-1.2ex]{0pt}{0pt} \rule{0pt}{2.5ex} &0 & -\sqrt{\frac{3}{2}}x_1+\sqrt{\frac{5}{2}}x_2 & x_3 & \sqrt{\frac{3}{2}}x_4+\sqrt{\frac{5}{2}}x_5\\
                   -\sqrt{2}x_1\rule[-1.2ex]{0pt}{0pt} \rule{0pt}{2.5ex} &\sqrt{\frac{3}{2}}x_1-\sqrt{\frac{5}{2}}x_2 & 0 & \sqrt{\frac{3}{2}}x_4-\sqrt{\frac{5}{2}}x_5& -2x_3 \\
                   -\sqrt{5}x_6\rule[-1.2ex]{0pt}{0pt} \rule{0pt}{2.5ex} & -x_3 &  -\sqrt{\frac{3}{2}}x_4+\sqrt{\frac{5}{2}}x_5 & 0 &  -\sqrt{\frac{3}{2}}x_1-\sqrt{\frac{5}{2}}x_2 \\
                   -\sqrt{2}x_4 &  -\sqrt{\frac{3}{2}}x_4-\sqrt{\frac{5}{2}}x_5 \rule[-1.2ex]{0pt}{0pt} \rule{0pt}{2.5ex} & 2x_3 &  \sqrt{\frac{3}{2}}x_1+\sqrt{\frac{5}{2}}x_2 & 0 \\
                  \end{smallmatrix}\right): x\in\R^7\right\}
\end{equation*}
Let $\{e_i\}$ be a $Q$-orthonormal basis of $\mathfrak{so}(5)$ such that $\{e_1,\dots,e_7\}$ correspond to the standard basis of $\m\cong\R^7$. The Lie brackets of elements of $\m$ are given below.
\begin{equation*}
\begin{array}{c | llllll}
[\cdot,\cdot] \rule[-1.2ex]{0pt}{0pt} \rule{0pt}{2.5ex}  & e_{2} & e_{3} & e_{4} & e_{5} &e_6 & e_7 \\
\hline
e_1\rule[-1.2ex]{0pt}{0pt} \rule{0pt}{2.5ex}  & e_7 & e_4+\sqrt6e_{10} &-e_3-e_9 & -e_6 & e_5-\sqrt{\tfrac52}e_{10} & e_2+\sqrt{\tfrac52}e_8 \\
e_2\rule[-1.2ex]{0pt}{0pt} \rule{0pt}{2.5ex}  & & -e_5 & -e_6 & e_3-3e_9 & e_4-\sqrt{\tfrac32}e_{10} & -e_1-\sqrt{\tfrac32}e_8 \\
e_3\rule[-1.2ex]{0pt}{0pt} \rule{0pt}{2.5ex}  & & &  e_1-\sqrt6e_8 &-e_2 & e_7 & -e_6 \\
e_4\rule[-1.2ex]{0pt}{0pt} \rule{0pt}{2.5ex}  & & & & -e_7 & -e_2+\sqrt{\tfrac52}e_8 & e_5+\sqrt{\tfrac52}e_{10}\\
e_5\rule[-1.2ex]{0pt}{0pt} \rule{0pt}{2.5ex}  & & & & &-e_1-\sqrt{\tfrac32}e_8 & -e_4+\sqrt{\tfrac32}e_{10} \\
e_6\rule[-1.2ex]{0pt}{0pt} \rule{0pt}{2.5ex}  & & & & & & e_3+2e_9\\
\end{array}
\end{equation*}
From \eqref{eq:rhatb7}, the kernel of $\widehat R\colon\Lambda^2\m\to\Lambda^2\m$ is spanned by the following $11$ vectors:
\begin{align*}
&e_4\wedge e_5+e_6\wedge e_9, &e_4\wedge e_6&+3\sqrt{\tfrac35}(e_4\wedge e_9+e_5\wedge e_6)-e_5\wedge e_{9},\\
&e_4\wedge e_5-e_7\wedge e_8, &e_4\wedge e_6&+3\sqrt{\tfrac35}(e_4\wedge e_9+e_5\wedge e_6)+e_8\wedge e_{10},\\
&e_4\wedge e_8-e_5\wedge e_{7}, &e_4\wedge e_{10}&+\tfrac13\sqrt{\tfrac53}(e_5\wedge e_{10}+e_6\wedge e_{7})+e_6\wedge e_8,\\
&e_4\wedge e_8-e_6\wedge e_{10}, &e_4\wedge e_{10}&+\tfrac23\sqrt{\tfrac53}(e_5\wedge e_{10}+e_6\wedge e_{7})+e_7\wedge e_9,\\
&e_5\wedge e_{10}-e_8\wedge e_9, &e_4\wedge e_9&+2e_5\wedge e_6 +e_7\wedge e_{10},\\
&\tfrac54e_4\wedge e_7+\tfrac14e_5\wedge e_8+e_9\wedge e_{10}. 
\end{align*}
Consider the self-adjoint operator induced by the $\H$-invariant $4$-form $\omega\in\Lambda^4\m$,
\begin{align*}
\omega&=e_1\wedge e_2\wedge (e_3\wedge e_6-e_4\wedge e_5)+(e_1\wedge e_4-e_2\wedge e_5)\wedge e_6\wedge e_7\\
&\quad+ e_1\wedge e_3\wedge e_5\wedge e_7+ e_2\wedge e_3\wedge e_4\wedge e_7+ e_3\wedge e_4\wedge e_5\wedge e_6.
\end{align*}
The restriction $\omega\colon\ker\widehat R\to\ker\widehat R$ is the identity operator, hence positive-definite. From Lemma~\ref{lemma:firstorder}, there exists $\varepsilon>0$ such that $(\widehat R+\varepsilon\omega)\colon \Lambda^2\m\to\Lambda^2\m$ is positive-definite, proving that the homogeneous space $(B^{7},\g)$ has strongly positive curvature.

\subsection{\texorpdfstring{Berger space $B^{13}$}{Berger space B13}}
The Berger space $B^{13}$ is the total space of a homogeneous fibration \eqref{eq:homfib}, where the Lie groups $\H\subset\K\subset\G$ are
\begin{equation*}
\Sp(2)\cdot\mathsf S^1 \subset \U(4)\subset\SU(5).
\end{equation*}
The inclusion $\U(4)\subset\SU(5)$ is given by $\U(4)=\{\diag(A,\det\overline{A})\in \SU(5):A\in\U(4)\}$. The inclusion $\Sp(2)\subset\U(4)$ is the usual one, and $\mathsf S^1\subset\U(4)$ is its center, formed by multiples of the identity matrix. We thus identify the corresponding homogeneous fibration as
\begin{equation*}
\R P^5\longrightarrow B^{13}\longrightarrow \C P^4.
\end{equation*}
The base $\C P^4$ is a CROSS, the fiber $\R P^5$ has constant positive curvature, and $(\k,\h)$ is a symmetric pair. The only remaining hypothesis to apply Theorem~\ref{thm:wallach} is strong fatness, for which we must analyze the appropriate Lie brackets.

The $Q$-orthogonal complement of $\h=\mathfrak{sp}(2)\oplus\R$ in $\mathfrak g=\mathfrak{su}(5)$ is:
\begin{equation*}
\C^4\oplus (\R\oplus\C^2)\ni (z,x,w)\longmapsto\left(\begin{matrix}
                    \frac{xi}{\sqrt{2}}\rule[-1.2ex]{0pt}{0pt} \rule{0pt}{2.5ex} & 0 & \frac{w_1}{\sqrt2} & \frac{w_2}{\sqrt2} & z_1 \\
                    0\rule[-1.2ex]{0pt}{0pt} \rule{0pt}{2.5ex} & \frac{xi}{\sqrt{2}} & \frac{\overline{w_2}}{\sqrt2} & -\frac{\overline{w_1}}{\sqrt2} & z_2 \\
                    -\frac{\overline{w_1}}{\sqrt{2}}\rule[-1.2ex]{0pt}{0pt} \rule{0pt}{2.5ex} & -\frac{w_2}{\sqrt2} & -\frac{xi}{\sqrt{2}} & 0 & z_3 \\
                    -\frac{\overline{w_2}}{\sqrt{2}}\rule[-1.2ex]{0pt}{0pt} \rule{0pt}{2.5ex} & \frac{w_1}{\sqrt{2}} & 0 & -\frac{xi}{\sqrt{2}} & z_4 \\
                    -\overline{z_1} & -\overline{z_2}\rule[-1.2ex]{0pt}{0pt} \rule{0pt}{2.5ex} & -\overline{z_3} & -\overline{z_4} & 0 \\
                  \end{matrix}\right)\in \m\oplus\p\subset\mathfrak{su}(5).
\end{equation*}
Let $\{e_i\}$ be a $Q$-orthonormal basis of $\mathfrak{su}(3)$ such that $\{e_1,\dots,e_8\}$ and $\{e_9,\dots,e_{13}\}$ correspond respectively to the standard basis of $\m\cong\C^4$ and $\p\cong\R\oplus\C^2$. The linear maps $L$ and hence $F=L^*L$, see \eqref{eq:L} and \eqref{eq:F}, are determined by the following table of Lie brackets, which are rescaled by $\sqrt2$ for convenience:
\begin{equation*}
\begin{array}{c | r r r r r}
\sqrt2\,[\cdot,\cdot]& e_9\rule[-1.2ex]{0pt}{0pt} \rule{0pt}{2.5ex}  & e_{10} & e_{11} & e_{12} & e_{13} \\
\hline
e_1\rule[-1.2ex]{0pt}{0pt} \rule{0pt}{2.5ex}  & -e_2 & e_5  & -e_6 & e_7 & -e_8 \\
e_2\rule[-1.2ex]{0pt}{0pt} \rule{0pt}{2.5ex}  & e_1 & e_6  & e_5 & e_8 & e_7\\
e_3\rule[-1.2ex]{0pt}{0pt} \rule{0pt}{2.5ex}  & -e_4 & -e_9  & -e_{10} & e_5 & e_6\\
e_4\rule[-1.2ex]{0pt}{0pt} \rule{0pt}{2.5ex}  & e_3 & -e_8  & e_7 & e_6 & -e_5\\
e_5\rule[-1.2ex]{0pt}{0pt} \rule{0pt}{2.5ex}  & e_6 & -e_1  & -e_2 & -e_3 & e_4\\
e_6\rule[-1.2ex]{0pt}{0pt} \rule{0pt}{2.5ex}  & -e_5 & -e_2  & e_1 & -e_4 & -e_3\\
e_7\rule[-1.2ex]{0pt}{0pt} \rule{0pt}{2.5ex}  & e_8 & e_3  & -e_4 & -e_1 & -e_2\\
e_8\rule[-1.2ex]{0pt}{0pt} \rule{0pt}{2.5ex}  & -e_7 & e_4  & e_3 & -e_2 & e_1\\
\end{array}
\end{equation*}
This implies that $\ker F=\ker L$ is spanned by the following $32$ vectors of $\m\otimes\p$:
\begin{align*}
-e_2\wedge e_9&+e_8\wedge e_{13}, & -e_1\wedge e_{9}&+e_8\wedge e_{12}, & -e_4\wedge e_{9}&+ e_8\wedge e_{11}, \\
 e_3\wedge e_{9}&+e_8\wedge e_{10}, & e_1\wedge e_{12}&+e_8\wedge e_{9}, & -e_1\wedge e_9&+e_7\wedge e_{13}, \\
 e_2\wedge e_{9}&+ e_7\wedge e_{12}, & -e_3\wedge e_{9}&+e_7\wedge e_{11}, & -e_4\wedge e_{9}&+e_7\wedge e_{10}, \\
 e_1\wedge e_{13}&+e_7\wedge e_{9}, & e_4\wedge e_9&+ e_6\wedge e_{13}, & -e_3\wedge e_{9}&+e_6\wedge e_{12},\\
-e_2\wedge e_9&+e_6\wedge e_{11}, & -e_1\wedge e_{9}&+e_6\wedge e_{10}, & e_1\wedge e_{10}&+ e_6\wedge e_{9},\\
e_3\wedge e_{9}&+e_5\wedge e_{13}, & e_4\wedge e_{9}&+e_5\wedge e_{12}, & -e_1\wedge e_9&+e_5\wedge e_{11}, \\
e_2\wedge e_{9}&+ e_5\wedge e_{10}, & e_1\wedge e_{11}&+e_5\wedge e_9, & e_1\wedge e_{10}&+e_4\wedge e_{13}, \\
e_1\wedge e_{11}&+e_4\wedge e_{12}, & -e_1\wedge e_{12}&+ e_4\wedge e_{11}, & -e_1\wedge e_{13}&+e_4\wedge e_{10},\\
e_1\wedge e_{11}&+e_3\wedge e_{13}, & -e_1\wedge e_{10}&+e_3\wedge e_{12}, & -e_1\wedge e_{13}&+ e_3\wedge e_{11}, \\
e_1\wedge e_{12}&+e_3\wedge e_{10}, & -e_1\wedge e_{12}&+e_2\wedge e_{13}, & e_1\wedge e_{13}&+e_2\wedge e_{12}, \\
-e_1\wedge e_{10}&+ e_2\wedge e_{11}, & e_1\wedge e_{11}&+e_2\wedge e_{10}. & &
\end{align*}
Consider the self-adjoint operator induced by the $\H$-invariant $4$-form $\tau\in\Lambda^2\m\otimes\Lambda^2\p$,
\begin{align*}
\tau&=-(e_3\wedge e_4-e_1\wedge e_2)\otimes(e_{10}\wedge e_{11}+e_{12}\wedge e_{13})\\
&\quad +(e_1\wedge e_3+e_2\wedge e_4)\otimes(e_{10}\wedge e_{12}-e_{11}\wedge e_{13})\\
&\quad -(e_1\wedge e_4-e_2\wedge e_3)\otimes(e_{10}\wedge e_{13}+e_{11}\wedge e_{12})\\
&\quad -(e_5\wedge e_6-e_7\wedge e_8)\otimes(e_{10}\wedge e_{11}-e_{12}\wedge e_{13})\\
&\quad +(e_5\wedge e_7+e_6\wedge e_8)\otimes(e_{10}\wedge e_{12}+e_{11}\wedge e_{13})\\
&\quad -(e_5\wedge e_8-e_6\wedge e_7)\otimes(e_{10}\wedge e_{13}-e_{11}\wedge e_{12})\\
&\quad +(e_1\wedge e_6-e_2\wedge e_5-e_3\wedge e_8+e_4\wedge e_7)\otimes (e_9\wedge e_{10})\\
&\quad +(e_1\wedge e_5+e_2\wedge e_6+e_3\wedge e_7+e_4\wedge e_8)\otimes (e_9\wedge e_{11})\\
&\quad +(e_1\wedge e_8-e_2\wedge e_7+e_3\wedge e_6-e_4\wedge e_5)\otimes (e_9\wedge e_{12})\\
&\quad +(e_1\wedge e_7+e_2\wedge e_8-e_3\wedge e_5-e_4\wedge e_6)\otimes (e_9\wedge e_{13}).
\end{align*}
The restriction $\tau\colon\ker F\to\ker F$ is the identity operator, hence positive-definite. From Lemma~\ref{lemma:firstorder}, there exists $\varepsilon>0$ such that $(F+\varepsilon\tau)\colon \m\otimes\p\to\m\otimes\p$ is positive-definite, proving strong fatness. Thus, by Theorem~\ref{thm:wallach}, the homogeneous space $(B^{13},\g_t)$ has strongly positive curvature for all $0<t<1$.

\smallskip

The above concludes the proof of Theorem~\ref{thm:strongposhomsp}, and hence of Theorem~\ref{mainthm:homspaces}.\qed

\begin{remark}
It was previously observed by P\"uttmann~\cite{puttmann,puttmann2} that $W^6$, $W^7_{1,1}$ and $B^7$ have strongly positive curvature, by directly computing their modified curvature operators. These computations were used to study optimal pinching constants of these homogeneous spaces. Although our method using Theorem~\ref{thm:wallach} is computationally simpler, it also provides less information. In particular, it does not allow to compute extremal sectional curvatures.
\end{remark}

\begin{remark}\label{rem:b13}
The normal homogeneous metric (corresponding to $t=1$) on the Berger space $B^{13}$ is known to have $\sec>0$ \cite{berger}. A direct computation of its modified curvature operator shows that this metric does not have strongly positive curvature, see also Remark~\ref{rem:zoltek} and Appendix~\ref{sec:berger}.
\end{remark}

\begin{remark}
Some of the above homogeneous spaces are related via submersions or totally geodesic immersions, allowing for alternative proofs that these spaces have strongly positive curvature. For instance, there is an embedding $W^6\to W^{12}$ whose image is the fixed-point set of an isometry, hence totally geodesic. In particular, it follows from Proposition~\ref{prop:immersions} that since $W^{12}$ has strongly positive curvature, so does $W^6$. Analogously, by the Taimanov embedding $W^7_{1,1}\to B^{13}$, the Aloff-Wallach space $W^7_{1,1}$ has strongly positive curvature since $B^{13}$ does. Finally, there are Riemannian submersions $W^7_{k,\ell}\to W^6$, so Theorem~\ref{thm:submersions} provides yet another proof that $W^6$ has strongly positive curvature.
\end{remark}

\section{Some Remarks and Open Questions}
\label{sec:openq}

The study of strongly positive curvature naturally leads to many interesting questions that are not addressed in the present paper. As a general principle, any question on manifolds with $\sec>0$ can be reformulated for manifolds with strongly positive curvature. In many cases, this yields a relevant problem, such as Problem~\ref{prob:stronghopf}. In dimension $4$, it gives an alternative viewpoint on problems about $\sec>0$, such as the Hopf question on $S^2\times S^2$, recall Proposition~\ref{prop:dim4}. In this section, we compile a few open questions in addition to these, organized under four common themes, commenting on how they relate to the literature on the subject.

\subsection{Topological obstructions}
There are notoriously few known topological obstructions to $\sec>0$, see \cite{zillersurvey}. On the other hand, manifolds with positive curvature operator $R>0$ are extremely rigid, in that their universal cover must be diffeomorphic to a sphere \cite{bw}. Since strongly positive curvature is an intermediate condition between $\sec>0$ and $R>0$, it is particularly relevant to ask:

\begin{problem}\label{prob:obstructions}
Are there topological obstructions to strongly positive curvature beyond those known for $\sec>0$?
\end{problem}

One tool to find topological obstructions to curvature conditions is Hamilton's Ricci flow. In particular, Ricci flow was used by B\"ohm and Wilking~\cite{bw} in the classification of manifolds with $R>0$, who proved that this condition is preserved under the flow. In earlier work, B\"ohm and Wilking~\cite{bw-ricci} provided the first example of a closed manifold with $\sec>0$ that develops mixed Ricci curvature under the Ricci flow. This example is the Wallach flag manifold $(W^{12},\g_*)$ with a homogeneous metric $\g_*$ that can be chosen \emph{arbitrarily close} to a metric of the form $\g_t$, discussed in Subsection~\ref{sec:w12}. In particular, $\g_*$ can be chosen to have strongly positive curvature, since this is an open condition. Therefore, \emph{strongly positive curvature is not preserved under the Ricci flow}.

Another natural attempt to find topological obstructions to strongly positive curvature is related to the Gauss-Bonnet integrand of a curvature operator, and the so-called \emph{algebraic Hopf conjecture}. Given an algebraic curvature operator $R\colon\Lambda^2 V\to\Lambda^2V$, $\dim V=2n$, its \emph{Gauss-Bonnet integrand} $\chi(R)$ is given by:
\begin{equation*}
\chi(R)=\sum_{\sigma,\tau\in \mathsf S_{2n}} \operatorname{sgn}(\sigma)\operatorname{sgn}(\tau)\prod_{i=1}^{2n-1} \big\langle R(e_{\sigma(i)}\wedge e_{\sigma(i+1)}),e_{\tau(i)}\wedge e_{\tau(i+1)}\big\rangle,
\end{equation*}
where $\mathsf S_{2n}$ is the group of permutations of $2n$ symbols, $\operatorname{sgn}(\sigma)$ denotes the sign of the permutation $\sigma$ and $\{e_i\}$ is an orthonormal basis of $V$. It has been long known that if $R>0$, then $\chi(R)>0$, see \cite[p. 191]{kulk}.
By the Chern-Gauss-Bonnet Theorem, the integral over a closed manifold $(M^{2n},\g)$ of $\chi(R_p)$ equal to the Euler characteristic $\chi(M)$, multiplied by a (positive) dimensional constant. A classic conjecture of Hopf asks whether even dimensional closed manifolds with $\sec>0$ have positive Euler characteristic. This conjecture remains open in general, whereas its \emph{algebraic} variant, that asks if an algebraic curvature operator $R$ in even dimensions with $\sec_R>0$ (recall \eqref{eq:secr}) has $\chi(R)>0$, is completely settled. Milnor proved in unpublished work that the algebraic Hopf conjecture holds in dimensions $\leq4$, see \cite{chern}. In particular, in light of Proposition~\ref{prop:dim4}, this shows that curvature operators in dimensions $\leq4$ that have strongly positive curvature also have $\chi(R)>0$. Geroch~\cite{ger}, and later Klembeck~\cite{klem}, provided counter-examples to the algebraic Hopf conjecture in dimensions $\geq6$, which are algebraic curvature operators in even dimensions $\geq6$ that have $\sec_R>0$ but $\chi(R)\leq0$. It is not hard to verify that some of these examples \emph{also have strongly positive curvature}, hence are also counter-examples to the strong version of the algebraic Hopf conjecture. Nevertheless, the following is completely open in $\dim\geq6$:

\begin{problem}\label{prob:stronghopf}
Does every closed manifold with even dimension and strongly positive curvature have $\chi(M)>0$?
\end{problem}

\subsection{Examples}
One of the largest problems in the study of manifolds with $\sec>0$ is the lack of examples. The only examples different from spheres and projective spaces currently known occur in dimensions $\leq24$ and have many symmetries, in the spirit of Grove's program~\cite{grove-survey}.

In the above sections, we proved that, except for $\Ca P^2=\mathsf F_4/\Spin(9)$ and possibly $W^{24}=\mathsf F_4/\Spin(8)$, the remaining closed simply-connected homogeneous spaces with $\sec>0$ have strongly positive curvature. We now make some remarks on these exceptional cases $\Ca P^2$ and $W^{24}$. Since the only $\mathsf F_4$-invariant metric on $\Ca P^2$ is the metric that makes $\Ca P^2$ a symmetric space, Proposition~\ref{prop:cross2} implies that this manifold does not admit homogeneous metrics with strongly positive curvature. Nevertheless, the following remains open:

\begin{problem}
Are there any (non-homogeneous) metrics on $\Ca P^2$ with strongly positive curvature?
\end{problem}

The $\mathsf F_4$-homogeneous metrics on $W^{24}$ depend on three positive numbers $x_1$, $x_2$ and $x_3$, corresponding to the scalings of the bi-invariant metric $Q$ on each of the three $\Spin(8)$-irreducible subspaces $\m=V_1\oplus V_2\oplus V_3$, where $\mathfrak f_4=\mathfrak{so}(8)\oplus\m$ is a $Q$-orthogonal splitting. The subspaces $V_j$ are $8$-dimensional and correspond to the three spinor representations $\R^8$, $\Delta_8^+$ and $\Delta_8^-$. If two of the three parameters $x_1,x_2,x_3$ are equal, then the corresponding homogeneous metric is such that the homogeneous fibration $S^8(\tfrac12)\to W^{24}\to\Ca P^2$, corresponding to the groups $\Spin(8)\subset\Spin(9)\subset \mathsf F_4$, is a Riemannian submersion, cf. \eqref{eq:homfib}. Thus, Theorem~\ref{thm:submersions} implies that \emph{such metrics do not have strongly positive curvature}. Nevertheless, there are clearly other homogeneous metrics on $W^{24}$ with $\sec>0$, which could possibly have strongly positive curvature.\footnote{A computation shows that there are three linearly independent $\mathsf F_4$-invariant $4$-forms in $W^{24}$, one in each of $\Lambda^2 V_i\otimes\Lambda^2 V_j$, $1\leq i<j\leq 3$.}

\begin{problem}\label{prob:w24}
Are there any $\mathsf F_4$-invariant metrics on $W^{24}$ with strongly positive curvature?
\end{problem}

An answer to Problem~\ref{prob:w24} would complete the classification of closed simply-connected homogeneous spaces with strongly positive curvature, initiated in the present paper. This approach uses the classification of closed simply-connected homogeneous spaces with $\sec>0$ (Theorem~\ref{thm:homsp}), which makes it natural to ask, in connection to Problem~\ref{prob:obstructions}:

\begin{problem}\label{prob:biquotients}
Is there a way to classify closed simply-connected homogeneous spaces with strongly positive curvature that is independent of Theorem~\ref{thm:homsp}?
\end{problem}

Apart from homogeneous spaces, other examples of closed manifolds with $\sec>0$ are given by \emph{biquotients}. A biquotient $\G/\!/\H$ is the orbit space of a free isometric action of a Lie group $\H\subset\G\times \G$ on a compact Lie group $\G$, given by $(h_1,h_2)\cdot g=h_1gh_2^{-1}$. The quotient map $\G\to\G/\!/\H$ is a Riemannian submersion, and hence formula \eqref{eq:oneill} can be applied to compute the curvature operator of biquotients. In particular, it can be used on the known examples of biquotients with $\sec>0$: the Eschenburg spaces $E^6$ and $E^7_{k,\ell}$ and the Bazaikin spaces $B^{13}_q$, which are respectively generalizations of $W^6$, $W^7_{k,\ell}$ and $B^{13}$, see \cite{zillersurvey}. This leads to our next:

\begin{problem}
Do the biquotients $E^6$, $E^7_{k,\ell}$ and $B^{13}_q$ have strongly positive curvature?
\end{problem}

Homogeneous spaces and biquotients aside, the only other currently known example\footnote{Apart from the proposed positively curved exotic sphere of Petersen and Wilhelm \cite{pw}.} of closed manifold with $\sec>0$ is a cohomogeneity one manifold. A manifold $(M,\g)$ is said to have \emph{cohomogeneity one} if it admits an isometric action by a Lie group $\G$ such that the orbit space $M/\G$ is one-dimensional. Grove, Wilking and Ziller~\cite{gwz} performed a systematic study of simply-connected closed cohomogeneity one manifolds with $\sec>0$, which lead to a classification result with an infinite family of \emph{candidate} manifolds. More precisely, they found two infinite families $P_k^7$, $Q_k^7$, $k\geq2$, and an exceptional case $R^7$, of $7$-dimensional manifolds which are the only manifolds different from homogeneous spaces and biquotients that could carry a cohomogeneity one metric with $\sec>0$.\footnote{Recently, Verdiani and Ziller~\cite{vzR} showed that $R^7$ does not admit an invariant metric with $\sec>0$.} Shortly after, Grove, Verdiani and Ziller~\cite{p2} constructed an invariant metric \emph{with strongly positive curvature} on the candidate $P^7_2$, which was also identified as an exotic $T_1S^4$. Dearricott~\cite{de} has independently found a metric with $\sec>0$ on this manifold, using modified curvature operators in an indirect way. These recent developments in cohomogeneity one suggest the following:

\begin{problem}
Classify closed simply-connected cohomogeneity one manifolds with strongly positive curvature (possibly independently of the $\sec>0$ classification by Grove, Wilking and Ziller~\cite{gwz}). In particular, can the remaining candidates $P^7_k$, $k\geq3$, and $Q^7_k$, $k\geq2$, support invariant metrics with strongly positive curvature?
\end{problem}

\subsection{``Best" modified curvature operators}
Given an algebraic curvature operator $R\colon\Lambda^2 V\to\Lambda^2 V$ with strongly positive curvature, there are many $4$-forms $\omega$ such that $R+\omega$ is positive-definite. Since, in general, there is no canonical choice $\omega_R$ to modify a particular $R$, we are led to the following intentionally vague:

\begin{problem}
Is there a ``best" $\omega_R\in\Lambda^4 V$ such that $R+\omega_R$ is positive-definite?
\end{problem}

The ambiguity in the notion of ``best" allows for many interpretations, centered around $\omega_R$ capturing the most information about $R$, or being determined in a canonical way $R\mapsto\omega_R$. Since the set of $\omega$'s such that $R+\omega$ is positive-definite is bounded and convex, it has a \emph{center of mass} $\omega_\text{CM}$, which is a natural candidate for ``best" $\omega$. In particular, notice that it varies continuously with $R$, by the Dominated Convergence Theorem.

An important application of such a canonical association $R\mapsto\omega_R$ would be that Sylvester's criterion applied to $R+\omega_R$ would yield a \emph{quantifier-free} description of strongly positive curvature. In other words, it would give a finite number of conditions on the entries of an algebraic curvature operator $R$ for it to have strongly positive curvature. It should be noted that, independently of the existence of a canonical association $R\mapsto\omega_R$, the set of algebraic curvature operators with strongly positive curvature is semi-algebraic by the Tarski-Seidenberg theorem, and hence there exists a finite decision procedure to determine whether a given curvature operator has strongly positive curvature, see \cite{mendes,weinstein}.

\subsection{Strongly nonnegative curvature}
Analogously to strongly positive curvature, we say that an algebraic curvature operator $R\colon\Lambda^2V\to\Lambda^2V$ has strongly \emph{nonnegative} curvature if there exists $\omega\in\Lambda^4V$ such that $R+\omega$ is positive-\emph{semidefinite}. From \eqref{eq:secrw}, if $R$ has strongly nonnegative curvature, then $\sec_R\geq0$, and the corresponding notion for manifolds is established in the same pointwise manner as strongly positive curvature (recall Definition~\ref{def:strongpos}). Most results in Section~\ref{sec:basics} are automatically valid for strongly nonnegative curvature, in particular, Riemannian submersions and Cheeger deformations preserve this condition.

However, note that the openness argument in Remark~\ref{rem:smoothness} does not apply in the case of strongly nonnegative curvature, which leads us to:

\begin{problem}\label{prob:nonneg}
Let $(M,\g)$ be a smooth manifold with strongly nonnegative curvature. Does there exist a \emph{smooth} $\omega\in\Omega^4(M)$ such that $R+\omega$ is positive semi-definite?
\end{problem}

In the special case of $\dim V=4$, a direct application of Proposition~\ref{prop:dim4} and \cite[Thm 2.1]{thorpeJDG} proves the following:

\begin{proposition}\label{prop:dim4nonneg}
Let $R\colon\Lambda^2V\to\Lambda^2V$ be an algebraic curvature operator, with $\dim V=4$. Assume that $\sec_R\geq0$ and that there exists $\sigma\in\Gr(V)$ such that $\sec_R(\sigma)=0$. Then there exists a \emph{unique} $\omega\in\Lambda^4 V$ such that $R+\omega$ is positive-semidefinite.
\end{proposition}

By the above, if $(M^4,\g)$ is a smooth manifold with strongly nonnegative curvature, then any $\omega$ such that $R+\omega$ is positive-semidefinite is completely determined on the subset $\mathcal Z=\{p\in M:\sec(\sigma)=0 \text{ for some } \sigma\subset T_pM\}$. In this situation, Problem \ref{prob:nonneg} is related to whether such $\omega\in\Omega^4(\mathcal Z)$ is smooth. Since any $\omega\in\Omega^4(M)$ is of the form $\omega=f\,\vol_M$, smoothness of $\omega$ is the same as smoothness of $f\colon \mathcal Z\to\R$.

It should be noted that, by the proof of Theorem~\ref{thm:submersions}, if $\pi\colon\overline M\to M$ is a Riemannian submersion and $\overline M$ has strongly nonnegative curvature with a \emph{smooth} modifying $4$-form, then the same holds for $M$. For instance, consider the $S^1$-action on $S^3\times S^2$ given by the Hopf action on $S^3$ and rotation on $S^2$. The quotient map $\pi\colon S^3\times S^2\to\C P^2\#\overline{\C P}^2$ is a Riemannian submersion, where the base $M^4=\C P^2\#\overline{\C P}^2$ has a cohomogeneity one metric with $\sec\geq0$, and \emph{every point} supports a plane with zero curvature, i.e., $\mathcal Z=M$, see \cite[p. 12]{bettiol}. In this case, the unique $4$-form $\omega\in\Omega^4(M)$ such that $R+\omega$ is positive-semidefinite is smooth, and can be computed as $\omega=r(2-r^2)^{-3/2}\vol_M$, where $r(p)\in [-1,1]$ is the height of the $S^1$-orbit $\pi^{-1}(p)\subset S^2\subset\R^3$.

Finally, note that if $\dim V\geq5$ and $R\colon\Lambda^2V\to\Lambda^2V$ is an algebraic curvature operator with strong nonnegative curvature (but without strongly positive curvature), then the uniqueness of $\omega\in\Lambda^4 V$ such that $R+\omega$ is positive-semidefinite may fail. For instance, the curvature operator of $S^4\times S^1$ can be modified with any sufficiently small multiple of the volume form of $S^4$, remaining positive-semidefinite.

\appendix
\section{Berger spheres}
\label{sec:berger}

The so-called \emph{Berger metrics} on spheres give an interesting example that illustrates how strongly positive curvature is related to positive-definite curvature operator and $\sec>0$. These are metrics on the total space of the Hopf bundles
\begin{equation*}
S^1\longrightarrow S^{2n+1}\longrightarrow \C P^n \quad \mbox{ and }\quad S^3\longrightarrow S^{4n+3}\longrightarrow \Hr P^n,
\end{equation*}
obtained by scaling the round metric along vertical directions, that is, metrics of the form $\mathbf g_\lambda=\lambda\,\mathbf g|_\text{ver}+\mathbf g|_\text{hor}$, where $\mathbf g_1=\mathbf g|_\text{ver}+\mathbf g|_\text{hor}$ is the round metric. Recall that $\mathbf g_\lambda$ has $\sec>0$ if and only if $0<\lambda<\tfrac43$, see \cite{vz}. Note also that the metrics $\mathbf g_\lambda$, $0<\lambda<1$, are Cheeger deformations of the round metric and hence have strongly positive curvature by Theorem~\ref{thm:cheegerdef}.

Let us consider the case of $S^3\to S^7\to \Hr P^1$. By a direct computation, its curvature operator is positive-definite if and only if $\tfrac12<\lambda<\lambda_1$, where $\lambda_1\cong 1.202$ is the only real root of $p_1(\lambda)=8\lambda^3-16\lambda^2+11\lambda-4$. Furthermore, it has strongly positive curvature if and only if $0<\lambda<\lambda_2$, where $\lambda_2\cong 1.304$ is the only real root of $p_2(\lambda)=25\lambda^3-60\lambda^2+48\lambda-16$.
Therefore, the proper inclusions $(\tfrac12,\lambda_1)\subsetneq (0,\lambda_2)\subsetneq (0,\tfrac43)$ represent proper inclusions of the classes of metrics $\mathbf g_\lambda$ on $S^7$ with, respectively, $R>0$, strongly positive curvature and $\sec>0$.

Note that $S^7$ admits a totally geodesic embedding into $S^{4n+3}$, where both are equipped with the Berger metric $\mathbf g_\lambda$. Thus, Proposition~\ref{prop:immersions} implies that $(S^{4n+3},\mathbf g_\lambda)$, for all $n\geq1$ and $\lambda_2<\lambda<\tfrac43$, are examples of closed manifolds with $\sec>0$ that do not have strongly positive curvature, see Remarks~\ref{rem:zoltek} and \ref{rem:b13}.
% OLD VERSION, WITH t NOTATION
%\begin{remark}
%Due to the work of Verdiani and Ziller~\cite{vz}, it is known that $(S^7,\g_t)$ has $\sec>0$ if and only if $0<t<\tfrac83\cong 2.667$. A direct computation of the curvature operator of $\g_t$ shows that it is positive-definite if and only if $1<t<t_1$, where $t_1\cong 2.404$ is the only real root of $p_1(t)=2t^3-8t^2+11t-8$. Furthermore, modifying this curvature operator with an $\Sp(2)\times\Sp(1)$-invariant $4$-form\footnote{Any such $4$-form is equal to $a\vol+b\,\tau$, where $\tau$ is given above and $\vol=e_1\wedge e_2\wedge e_3\wedge e_4\in\Lambda^4\m$.} shows that $\g_t$ has strongly positive curvature if and only if $0<t<t_2$, where $t_2\cong 2.608$ is the only real root of $p_2(t)=25t^3-120t^2+192t-128$.
%
%Therefore, the proper inclusions $(1,t_1)\subsetneq (0,t_2)\subsetneq (0,\tfrac83)$ represent proper inclusions of classes of metrics $\g_t$ on $S^7$ with, respectively, $R>0$, strongly positive curvature and $\sec>0$. In particular, $(S^7,\g_t)$, $t_2<t<\tfrac83$, provides another example of a closed manifold with $\sec>0$ that does not have strongly positive curvature, cf. Remark~\ref{rem:zoltek}.
%\end{remark}

\end{document}